\documentclass{amsart}
\usepackage{amsmath}

\usepackage{mathrsfs}
\usepackage{amsfonts}
\usepackage{mathrsfs}
\usepackage{amssymb}
\usepackage{amsmath,amssymb}

\newtheorem{theorem}{Theorem}[section]
\newtheorem{lemma}[theorem]{Lemma}
\newtheorem{proposition}[theorem]{Proposition}

\theoremstyle{definition}
\newtheorem{definition}[theorem]{Definition}

\theoremstyle{remark}
\newtheorem{remark}[theorem]{Remark}

\numberwithin{equation}{section}

\begin{document}

\title[ Generalized Naiver-Stokes equations in local $Q$-type spaces]
{Generalized Naiver-Stokes equations with initial  data in local
$Q$-type spaces}

\author{Pengtao Li}
\address{School of Mathematics, Peking University, Beijing, 100871, China}
\curraddr{Department of Mathematics and Statistics, Memorial
University of Newfoundland, St. John's, NL A1C 5S7, Canada}
\email{li\_ptao@163.com}
\author{Zhichun Zhai}
\address{Department of Mathematics and Statistics, Memorial University of Newfoundland, St. John's, NL A1C 5S7, Canada}
\curraddr{}
 \email{a64zz@mun.ca}
\thanks{Project supported in part  by Natural Science and
Engineering Research Council of Canada.}


\date{}

\dedicatory{}


\begin{abstract}
In this paper, we establish a link between Leray mollified solutions
of the three-dimensional generalized Naiver-Stokes equations and
mild solutions for initial data in the adherence of the test
 functions for the norm of $Q^{\beta,-1}_{\alpha, \ loc}(\mathbb{R}^{3}).$
 This result  applies to the usual incompressible Navier-Stokes
 equations and deduces a known  link.
\end{abstract}
\maketitle

\section{Introduction}
This paper studies the relationship between  Leray mollified
solutions and mild solutions to the generalized Naiver-Stokes
equations in $\mathbb{R}^{3}:$
\begin{equation}\label{eq1e}
\left\{\begin{array} {l@{\quad \quad}l}
\partial_{t}u+(-\triangle)^{\beta}u+(u\cdot \nabla)u-\nabla p=0,
& \hbox{in}\ \mathbb{R}^{1+3}_{+};\\
\nabla \cdot u=0, & \hbox{in}\ \mathbb{R}^{1+3}_{+};\\
u|_{t=0}=u_{0}, &\hbox{in}\ \mathbb{R}^{3}, \end{array} \right.
\end{equation}
for $\beta\in (1/2,1],$ where $(-\triangle)^{\beta}$ is the
fractional Laplacian with respect to $x$ defined by
$$
\widehat{(-\triangle)^{\beta}u}(t,\xi)
 =|\xi|^{2\beta}\widehat{u}(t,\xi)
 $$
through Fourier transform.
  Here $u$ and $p$ are non-dimensional
quantities corresponding to the velocity of the fluid and its
pressure. $u_{0}$ is the initial data and for the sake of
simplicity, the fluid is supposed to fill the whole space
$\mathbb{R}^{3}.$ Equations (\ref{eq1e}) have been studied
intensively, see Cannone \cite{M Cannone},   Giga and
Miyakawa\cite{Y Giga T Miyakawa}, Kato \cite{T Kato}, Koch and
Tataru \cite{H. Koch D. Tataru}, Xiao \cite{J. Xiao 1}, Lions
\cite{J L Lions}, Wu \cite{J Wu 2}-\cite{J Wu 3}, Li and Zhai
\cite{Li Zhai}.

When $\beta=1$,  equations (\ref{eq1e}) become the usual
incompressible Naiver-Stokes equations. In dimensional 3, the global
existence, uniqueness and regularity of the solutions for the usual
Naiver-Stokes equations are long-standing open problem of fluid
dynamics and the regularity problems is of course a millennium prize
problem. Generally speaking, there are two specific approaches in
the study of the existence of solutions to the three-dimensional
incompressible Naiver-Stokes equations. The first one is due to
Leray \cite{Leray} and the second is due to Kato \cite{T Kato}. We
refer the readers to Cannone \cite{Cannone} and
Lemari$\acute{e}$-Rieusset \cite{LR} for further information.

For general $\beta$, we can also define  the mild  and mollified
solutions separately as follows.

The generalized Naiver-Stokes system is  equivalent to the  fixed
point problem:
\begin{equation}\label{eq2}
u(t,x)=e^{-t(-\Delta)^{\beta}}u_{0}(x)-B(u,u)(t,x),
\end{equation}
where the bilinear form $B(u,v)$ is defined by
$$B(u,v)=\int_{0}^{t}e^{-(t-s)(-\Delta)^{\beta}}\mathbb{P}\nabla\cdot(u\otimes v)(s,x)ds.$$
Here $e^{-t(-\Delta)^{\beta}}$ denotes the convolution operator
generated by the symbol
$$\widehat{[e^{-t(-\Delta)^{\beta}}]}(\xi)=e^{-t|\xi|^{2\beta}}$$ and
$\mathbb{P}$ denotes the Leray projector onto the divergence free
vector filed.
\begin{definition}\label{Def1}
A mild solution to the generalized Naiver-Stokes equations
(\ref{eq1e}) is a solution to equations (\ref{eq2}) obtained via a
fixed point procedure.
\end{definition}

The mollified solutions are constructed in the same way as mild
solutions, but with a slightly different model. Indeed, instead of
the term $u\otimes u$ involved in the $(GNS)$ equations, we look for
something smoother. Let $\omega\in\mathcal{D}(\mathbb{R}^{3})$ with
$\omega>0$ and $\int_{\mathbb{R}^{3}}\omega(x)dx=1$. Then for
$\varepsilon>0$, the mollified generalized Navier-Stokes equations
are given by:
\begin{equation}\label{eq3}
\left\{\begin{array} {l@{\quad \quad}l}
\partial_{t}u+(-\triangle)^{\beta}u+((u\ast\omega_{\varepsilon})\cdot \nabla)u-\nabla p=0,
& \hbox{in}\ \mathbb{R}^{1+3}_{+};\\
\nabla \cdot u=0, & \hbox{in}\ \mathbb{R}^{1+3}_{+};\\
u|_{t=0}=u_{0}, &\hbox{in}\ \mathbb{R}^{3} \end{array} \right.
\end{equation}
with
$\omega_{\varepsilon}(x)=\frac{1}{\varepsilon^{3}}\omega(\frac{x}{\varepsilon}).$

Similar to the equations  (\ref{eq1e}), equations (\ref{eq3}) can be
rewritten as a fixed point problem:
\begin{equation}\label{eq4}
u_{\varepsilon}=e^{-t(-\Delta)^{\beta}}u_{0}-B_{\varepsilon}(u_{\varepsilon},u_{\varepsilon}),
\end{equation}
where the bilinear operator $B_{\varepsilon}$ is defined by
$$B_{\varepsilon}(u,u)=\int_{0}^{t}e^{-(t-s)(-\Delta)^{\beta}}\mathbb{P}\nabla\cdot((u\ast\omega_{\varepsilon})\otimes u)(s)ds.$$
\begin{definition}
The mollified solution to equations (\ref{eq1e}) is the sequence
$\{u_{\varepsilon}\}_{\varepsilon>0}$ of the solutions to the system
(\ref{eq4}) for $\varepsilon>0$.
\end{definition}

In \cite{LRP}, for $\beta=1$ of equations (\ref{eq1e}),  that is,
the usual incompressible Navier-Stokes equations,
Lemari$\acute{e}$-Rieusset and Prioux established a link between
these two solutions. They proved that if the initial data
$u_{0}\in\overline{\mathcal{D}(\mathbb{R}^{3})}^{bmo^{-1}(\mathbb{R}^{3})}$,
then there exists $T>0$ such that the mollified solutions
$u_{\varepsilon}
\in\overline{\mathcal{D}((0,T]\times\mathbb{R}^{3})}^{X^{1}_{0;T}(\mathbb{R}^{3})}$
constructed via the theory of Leray converges, when $\varepsilon$
tends to $0$, to the mild solution given by Kato, for $t\in (0,T)$.

 In \cite{Li Zhai}, inspired by Xiao's paper \cite{J. Xiao 1}, we considered
 the well-posedness and regularity of
equations (\ref{eq1e}) with initial data in some new critical spaces
$Q^{\beta, -1}_{\alpha; \infty}(\mathbb{R}^{n})$. In that paper, we
proved that for initial data $u_{0}\in Q^{\beta,-1}_{\alpha;
\infty}(\mathbb{R}^{n})$ there exists a unique mild  solution in the
space $X^{\beta}_{\alpha;\infty}$, where the space
$Q^{\beta,-1}_{\alpha;\infty}(\mathbb{R}^{n})$ occurring above is a
class of spaces which own a structure similar to the space
$BMO^{-1}(\mathbb{R}^{n})$ in \cite{H. Koch D. Tataru} and
$Q^{-1}_{\alpha; \infty}(\mathbb{R}^{n})$ in \cite{J. Xiao 1}. It is
easy to see that if $\alpha=-\frac{n}{2}$ and $\beta=1$,
$Q^{\beta,-1}_{\alpha;
\infty}(\mathbb{R}^{n})=BMO^{-1}(\mathbb{R}^{n})$, and if
$\alpha\in(0,1)$ and $\beta=1$, $Q^{\beta,-1}_{\alpha ;
\infty}(\mathbb{R}^{n})=Q^{-1}_{\alpha; \infty}(\mathbb{R}^{n}).$
Here $Q^{-1}_{\alpha; \infty}(\mathbb{R}^{n})$ is the derivative
space of $Q_{\alpha}(\mathbb{R}^{n}),$ see Xiao \cite{J. Xiao 1},
Dafni and Xiao \cite{G. Dafni J. Xiao}-\cite{G. Dafni J. Xiao 1},
Essen,  Janson, Peng  and  Xiao \cite{M. Essen S. Janson L. Peng J.
Xiao}. Therefore our well-posed result generalized the result of
Koch and Tataru \cite{H. Koch D. Tataru} and that of Xiao \cite{J.
Xiao 1}.

 The main goal of this paper is to  establish
 a relation between the mild solutions obtained  in \cite{J. Xiao 1} and   \cite{Li Zhai}
  and the mollified solutions for the
 equations (\ref{eq3}). In fact, our main results mean that
 for  initial data $u_{0}\in\overline{\mathcal{D}(\mathbb{R}^{3})}^{Q^{\beta,-1}_{\alpha,\
loc}(\mathbb{R}^{3})},$ with $\beta\in (1/2,1],$
 there exists $T>0$
such that the sequence $\{u_{\varepsilon}\}_{\varepsilon>0}
\in\overline{\mathcal{D}((0,T]\times\mathbb{R}^{3})}^{X^{\beta}_{\alpha,T}(\mathbb{R}^{3})}$
of solutions to (\ref{eq3}) converges, when
$\varepsilon\longrightarrow0,$  to the mild solution obtained in
\cite{J. Xiao 1} and  \cite{Li Zhai}, for $t\in (0,T).$ For the
usual incompressible Navier-Stokes equations,  when $\alpha=0,$ our
main result goes back to Lemari$\acute{e}$-Rieusset and Prioux
\cite[Theorem 1.1]{LRP}. However, it is worth pointing out that
their Theorem does not deduce our results even though
$Q^{1,-1}_{\alpha;\ loc}(\mathbb{R}^{3})$ is a subspace of
$bmo^{-1}(\mathbb{R}^{3}),$ since $X^{1}_{\alpha;T}(\mathbb{R}^{3})$
is proper subspace of $X^{1}_{0;T}(\mathbb{R}^{3})$ when
$0<\alpha<1.$

In the following, we give some definitions and known results. The
first one is  the space $Q^{\beta,-1}_{\alpha,loc}(\mathbb{R}^{n})$
defined as follows.

\begin{definition}
For $\alpha>0$ and
 $\max\{1/2,\alpha\}<\beta\leq 1$ with $\alpha+\beta-1\geq 0,$  $Q^{\beta,-1}_{\alpha,loc}(\mathbb{R}^{n})$
is the space of tempered distributions $f$ on $\mathbb{R}^{n}$ such
that, for all $T\in (0,\infty),$
$$\sup_{0<r^{2\beta}<T}\sup_{x_{0}\in\mathbb{R}^{n}}r^{2\alpha-n+2\beta-2}\int^{r^{2\beta}}_{0}
\int_{|x-x_{0}|<r}|e^{-t(-\Delta)^{\beta}}f(y)|^{2}\frac{dydt}{t^{\alpha/\beta}}<\infty.$$
The norm on $Q^{\beta,-1}_{\alpha,loc}(\mathbb{R}^{n})$ is defined
by
$$\|f\|_{Q^{\beta,-1}_{\alpha,loc}(\mathbb{R}^{n})}=\left(\sup_{0<r^{2\beta}<1}\sup_{x_{0}\in\mathbb{R}^{n}}r^{2\alpha-n+2\beta-2}\int^{r^{2\beta}}_{0}
\int_{|x-x_{0}|<r}|e^{-t(-\Delta)^{\beta}}f(y)|^{2}\frac{dydt}{t^{\alpha/\beta}}\right)^{1/2}.$$
\end{definition}
 In \cite{Li Zhai}, it was proved that that the space $Q^{\beta,-1}_{\alpha,loc}(\mathbb{R}^{n})$
consist of the derivatives of functions in
$Q^{\beta}_{\alpha}(\mathbb{R}^{n})$ which is composed of all
measurable functions with
$$\sup_{I}\ [l(I)]^{2(\alpha+\beta-1)-n}\int_{I}\int_{I}\frac{|f(x)-f(y)|^{2}}{|x-y|^{n+2(\alpha-\beta+1)}}dxdy<\infty$$
where the supremum is taken over all cubes $I$ with the edge length
$l(I)$ and the edges parallel to the coordinate axes in
$\mathbb{R}^{n}$.

We now introduce the space $X^{\beta}_{\alpha,T}(\mathbb{R}^{n}).$

\begin{definition}  \label{X space} Let $\alpha>0$ and
 $\max\{1/2,\alpha\}<\beta\leq 1$ with $\alpha+\beta-1\geq 0$.\\
(i)  A tempered distribution $f$ on $\mathbb{R}^{n}$ belongs to
$Q_{\alpha;T}^{\beta,-1}(\mathbb{R}^{n})$ provided
 $$\|f\|_{Q_{\alpha;T}^{\beta,-1}(\mathbb{R}^{n})}=\sup_{x\in\mathbb{R}^{n},r^{2\beta}\in(0,T)}\left(r^{2\alpha-n+2\beta-2}
\int_{0}^{r^{2\beta}}\int_{|y-x|<r}| K_{t}^\beta\ast
f(y)|^{2}t^{-\frac{\alpha}{\beta}}dydt\right)^{1/2}<\infty;
$$
 (ii) A tempered distribution $f$ on $\mathbb{R}^{n}$ belongs to
$\overline{VQ^{\beta,-1}_{\alpha}}(\mathbb{R}^{n})$ provided
$\lim\limits_{T\longrightarrow
0}\|f\|_{Q^{\beta,-1}_{\alpha;T}(\mathbb{R}^{n})}=0;$\\
 (iii) A function $g$ on $\mathbb{R}^{1+n}_{+}$ belongs to the space
$X^{\beta}_{\alpha;T}(\mathbb{R}^{n})$ provided
\begin{eqnarray*}
\|g\|_{X^{\beta}_{\alpha;T}(\mathbb{R}^{n})}&=&\sup_{t\in(0,T)}t^{1-\frac{1}{2\beta}}\|g(t,\cdot)\|_{L^{\infty}(\mathbb{R}^{n})}\\
&&+\sup_{x\in
\mathbb{R}^{n},r^{2\beta}\in(0,T)}\left(r^{2\alpha-n+2\beta-2}\int_{0}^{r^{2\beta}}\int_{|y-x|<r}|g(t,y)|^{2}t^{-\alpha/\beta}dydt\right)^{1/2}<\infty.
\end{eqnarray*}
\end{definition}

In Xiao \cite{J. Xiao 1} and Li and Zhai \cite{Li Zhai}, they proved
the following well-posedness results about the equations
(\ref{eq1e}) for $\beta=1$ and $\frac{1}{2}<\beta<1$, respectively.

\begin{theorem} 
Let $n\geq 2,$ $\alpha>0$ and $\max\{\alpha,1/2\}<\beta\leq1$ with $\alpha+\beta-1\geq0$. Then\\
(i) The generalized Navier-Stokes system (\ref{eq1e}) has a unique
small global mild solution in
$(X^{\beta}_{\alpha;\infty}(\mathbb{R}^{n}))^{n}$ for all initial
data $a$ with $\nabla\cdot a=0$ and
$\|a\|_{(Q_{\alpha;\infty}^{\beta,-1}(\mathbb{R}^{n}))^{n}}$ being small.\\
(ii) For any $T\in(0,\infty)$ there is an $\varepsilon>0$ such that
the generalized Navier-Stokes system (\ref{eq1e}) has a unique small
mild solution in $(X_{\alpha; T}^{\beta}(\mathbb{R}^{n}))^{n}$ on
$(0,T)\times \mathbb{R}^{n}$ when the initial data $a$ satisfies
$\nabla\cdot a=0$ and
$\|a\|_{(Q_{\alpha;T}^{\beta,-1}(\mathbb{R}^{n}))^{n}}\leq
\varepsilon.$ In particular for all $a\in
(\overline{VQ_{\alpha}^{\beta,-1}}(\mathbb{R}^{n}))^{n}$ with
$\nabla\cdot a=0$ there exists a unique small local mild solution in
$(X_{\alpha;T}^{\beta}(\mathbb{R}^{n}))^{n}$ on $(0,T)\times
\mathbb{R}^{n}.$
 \end{theorem}
\begin{remark}
The core of the proof of the above theorem is the following
inequality: for $u$ and $v\in X^{\beta}_{\alpha
;T}(\mathbb{R}^{n})$, we have
\begin{equation}\label{eq13}
\|B(u,v)\|_{X^{\beta}_{\alpha;T}(\mathbb{R}^{n})}\leq\|u\|_{X^{\beta}_{\alpha;T}(\mathbb{R}^{n})}\|v\|_{X^{\beta}_{\alpha;T}(\mathbb{R}^{n})}.
\end{equation}
The above inequality will play an important role in  this paper.
\end{remark}
 We recall the definition of the Lorentz space
$L^{p,\infty}(\mathbb{R}^{n})$ for $1<p<\infty.$
\begin{definition}
Let $1<p<\infty.$ A function $f\in L^{1}_{loc}(\mathbb{R}^{n})$ is
in the Lorentz space $L^{p,\infty}(\mathbb{R}^{n})$ if and only
if
$$\forall\lambda>0,\quad \left|\left\{x\in\mathbb{R}^{n}, |f(x)|>\lambda\right\}\right|\leq\frac{C}{\lambda^{p}}.$$
For $p=\infty$, we have
$L^{\infty,\infty}(\mathbb{R}^{n})=L^{\infty}(\mathbb{R}^{n}).$
\end{definition}
In \cite{LRP}, the authors introduced a new class of  Lorentz
spaces.
\begin{definition}
For $1<p<\infty$, the weak Lorentz space
$\widetilde{L}^{p,\infty}((0,T))$ is  the adherence of functions in
$L^{\infty}((0,T))$ for the norm in the Lorentz space
$L^{p,\infty}((0,T))$, that is,
$$\widetilde{L}^{p,\infty}((0,T))=\overline{L^{\infty}((0,T))}^{L^{p,\infty}((0,T))}.$$
\end{definition}

Let $\widetilde{L}^{p,\infty}((0,T), L^{q,\infty}(\mathbb{R}^{3}))$
be the space  of all measurable functions $f$ on
$(0,T)\times\mathbb{R}^{3}$ such that
$\|f(t,\cdot)\|_{L^{q,\infty}(\mathbb{R}^{3})}\in
\widetilde{L}^{p,\infty}((0,T)).$ Then the following proposition
holds.

\begin{proposition}\label{pro1}(\cite[Proposition 2.7 ]{LRP}) Let
$T>0$, $1<p<\infty$ and $1<q\leq\infty$. The following properties
are equivalent:

(1)\quad $f\in\widetilde{L}^{p,\infty}((0,T),
L^{q,\infty}(\mathbb{R}^{3}));$

(2)\quad For all $\lambda>0$, there exists a constant $C(\lambda)$
such that $C(\lambda)\longrightarrow 0$\quad$ \hbox{as}\
\lambda\longrightarrow\infty$ and
$$\left|\left\{\|f(t)\|_{L^{q,\infty}(\mathbb{R}^{3})}>\lambda\right\}\right|<\frac{C(\lambda)}{\lambda^{p}};$$

(3)\quad For all $\varepsilon>0$, there exists $f_{1}\in
L^{\infty}((0,T),L^{q,\infty}(\mathbb{R}^{3}))$ and $f_{2}\in
L^{p,\infty}((0,T),L^{q,\infty}(\mathbb{R}^{3}))$ such that
$$\|f_{2}\|_{L^{p,\infty}((0,T),L^{q,\infty}(\mathbb{R}^{3}))}\leq \varepsilon\quad\text{ and }\quad f=f_{1}+f_{2}.$$
\end{proposition}

The rest of this paper is organized as follows: In Section 2, we
give two technical lemmas: Lemma \ref{le1}-the continuity of the
bilinear operator $B(\cdot,\cdot)$ in  Lorentz spaces; Lemma
\ref{le1}- a local existence of mild solution to equations
(\ref{eq1e}) with initial data in Lorentz spaces.  In Section 3, we
establish main results of this paper. We only provide a proof for
three spatial dimensions, but our proof goes through almost verbatim
in higher dimensions.

 \vspace{0.1in}

\section{Technical Lemmas}

In this section, we prove two preliminary lemmas. The first one can
be regarded as an generalization of   \cite[Lemma 6.1]{LRP} for the
case $\beta=1.$

\begin{lemma}\label{le1}
Let $T>0$ and $\frac{1}{2}<\beta<1$,
$$B(u,v)(t,x)=\int_{0}^{t}e^{-(t-s)(-\Delta)^{\beta}}\mathbb{P}\nabla(u\otimes v)(s,x)ds.$$
Let $\frac{2\beta}{2\beta-1}\leq p<\infty$ and
$\frac{3}{2\beta-1}<q\leq\infty$ such that
$\beta-\frac{1}{2}=\frac{\beta}{p}+\frac{3}{2q}$. Then we have\\
(1)\quad for $p\neq\frac{2\beta}{2\beta-1}$ (and so $q\neq\infty$),
$$B:\quad L^{p,\infty}((0,T),L^{q,\infty}(\mathbb{R}^{3}))\times L^{p,\infty}((0,T),L^{q,\infty}(\mathbb{R}^{3}))\longrightarrow
L^{p,\infty}((0,T),L^{q,\infty}(\mathbb{R}^{3}))$$
with
$$\|B(u,v)\|_{L^{p,\infty}((0,T),L^{q,\infty}(\mathbb{R}^{3}))}\leq C\|u\|_{L^{p,\infty}((0,T),L^{q,\infty}(\mathbb{R}^{3}))}
\|v\|_{L^{p,\infty}((0,T),L^{q,\infty}(\mathbb{R}^{3}))};$$
 (2)\quad
$B: L^{\infty}((0,T),L^{q,\infty}(\mathbb{R}^{3}))\times
L^{p,\infty}((0,T),L^{q,\infty}(\mathbb{R}^{3}))\longrightarrow
L^{\infty}((0,T),L^{q,\infty}(\mathbb{R}^{3}))$ with
$$\|B(u,v)\|_{L^{\infty}((0,T),L^{q,\infty}(\mathbb{R}^{3}))}
\leq
C\|u\|_{L^{\infty}((0,T),L^{q,\infty}(\mathbb{R}^{3}))}\|v\|_{L^{p,\infty}((0,T),L^{q,\infty}(\mathbb{R}^{3}))};$$
(3)\quad $B: L^{\infty}((0,T),L^{q,\infty}(\mathbb{R}^{3}))\times
L^{\infty}((0,T),L^{q,\infty}(\mathbb{R}^{3}))\longrightarrow
L^{\infty}((0,T),L^{q,\infty}(\mathbb{R}^{3}))$ with
$$\|B(u,v)\|_{L^{\infty}((0,T),L^{q,\infty}(\mathbb{R}^{3}))}
\leq
CT^{1/p}\|u\|_{L^{\infty}((0,T),L^{q,\infty}(\mathbb{R}^{3}))}\|v\|_{L^{\infty}((0,T),L^{q,\infty}(\mathbb{R}^{3}))};$$
(4)\quad $B: L^{\infty}((0,T),L^{q,\infty}(\mathbb{R}^{3}))\times
L^{p,\infty}((0,T),L^{q,\infty}(\mathbb{R}^{3}))\longrightarrow
L^{p,\infty}((0,T),L^{q,\infty}(\mathbb{R}^{3}))$ with
$$\|B(u,v)\|_{L^{p,\infty}((0,T),L^{q,\infty}(\mathbb{R}^{3}))}
\leq
CT^{1/p}\|u\|_{L^{\infty}((0,T),L^{q,\infty}(\mathbb{R}^{3}))}\|v\|_{L^{p,\infty}((0,T),L^{q,\infty}(\mathbb{R}^{3}))}.$$
\end{lemma}
\begin{proof}
The proof of this lemma bases on the following inequality:
\begin{equation}\label{eq5}
\|B(u,v)(t)\|_{L^{q,\infty}(\mathbb{R}^{3})}\leq\int^{t}_{0}(t-s)^{\frac{1}{2\beta}(-1-\frac{3}{q})}\|u(s)\|_{L^{q,\infty}(\mathbb{R}^{3})}
\|v(s)\|_{L^{q,\infty}(\mathbb{R}^{3})}ds.
\end{equation}
 In fact, since
$$B(u,v)(t,x)=\int_{0}^{t}e^{-(t-s)(-\Delta)^{\beta}}\mathbb{P}\nabla(u\otimes v)(s,x)ds,$$
we have
$$\|B(u,v)(t)\|_{L^{q,\infty}(\mathbb{R}^{3})}
\leq\int_{0}^{t}\|e^{-(t-s)(-\Delta)^{\beta}}\mathbb{P}\nabla(u\otimes
v)(s,\cdot)\|_{L^{q,\infty}(\mathbb{R}^{3})}ds.$$ Since
$e^{-u(-\Delta)^{\beta}}\mathbb{P}\nabla$ is a convolution operator,
  the Young inequality tells us for
$1+\frac{1}{q}=\frac{2}{q}+\frac{1}{r},$
\begin{eqnarray*}
\|B(u,v)(t)\|_{L^{q,\infty}(\mathbb{R}^{3})}
&\leq&\int_{0}^{t}\|e^{-(t-s)(-\Delta)^{\beta}}\mathbb{P}\nabla\|_{L^{r,\infty}(\mathbb{R}^{3})}\|u\otimes
v(s)\|_{L^{q/2,\infty}(\mathbb{R}^{3})}ds\\
&\leq&\int_{0}^{t}\|e^{-(t-s)(-\Delta)^{\beta}}\mathbb{P}\nabla\|_{L^{r,\infty}(\mathbb{R}^{3})}
\|u(s)\|_{L^{q,\infty}(\mathbb{R}^{3})}\|v(s)\|_{L^{q,\infty}(\mathbb{R}^{3})}ds.
\end{eqnarray*}
For $n=3$, the derivation of the generalized Oseen kernel satisfies
$$\left|e^{-(t-s)(-\Delta)^{\beta}}\mathbb{P}\nabla(x,y)\right|\leq\frac{1}{((t-s)^{\frac{1}{2\beta}}+|x-y|)^{4}}$$
(See \cite[Lemma 4.10]{Li Zhai}).  Then we can get
\begin{eqnarray*}
\|e^{-(t-s)(-\Delta)^{\beta}}\mathbb{P}\nabla\|_{L^{r,\infty}(\mathbb{R}^{3})}&\leq&
\|e^{-(t-s)(-\Delta)^{\beta}}\mathbb{P}\nabla\|_{L^{r,\infty}(\mathbb{R}^{3})}\\
&\leq&\left(\int_{\mathbb{R}^{3}}\frac{1}{((t-s)^{\frac{1}{2\beta}}+|x-y|)^{4r}}dy\right)^{1/r}\\
&=&\left(\int_{0}^{\infty}(t-s)^{-\frac{4r}{2\beta}}\frac{(|x-y|/t^{\frac{1}{2\beta}})^{2}}{\left(1+t^{-\frac{1}{2\beta}}|x-y|\right)^{4r}}
t^{\frac{3}{2\beta}}d\left(t^{-\frac{1}{2\beta}}|x-y|\right)\right)^{1/r}\\
&=&(t-s)^{\frac{3}{2\beta
r}-\frac{4}{2\beta}}\left(\int_{0}^{\infty}\frac{\lambda^{2}}{(1+\lambda)^{4r}}d\lambda\right)^{1/r}\\
&\leq&C(t-s)^{\frac{1}{2\beta}(\frac{3}{r}-4)}.
\end{eqnarray*}
For $1=\frac{1}{r}+\frac{1}{q}$, we have
\begin{eqnarray*}
\|B(u,v)(t)\|_{L^{q,\infty}(\mathbb{R}^{3})}&\leq&\int^{t}_{0}(t-s)^{\frac{1}{2\beta}(\frac{3}{r}-4)}\|u(s)\|_{L^{q,\infty}(\mathbb{R}^{3})}
\|v(s)\|_{L^{q,\infty}(\mathbb{R}^{3})}ds\\
&=&\int^{t}_{0}(t-s)^{\frac{1}{2\beta}(-1-\frac{3}{q})}\|u(s)\|_{L^{q,\infty}(\mathbb{R}^{3})}\|v(s)\|_{L^{q,\infty}(\mathbb{R}^{3})}ds.
\end{eqnarray*}
This completes the proof of (\ref{eq5}). Now we prove $(1)-(4)$ by
using  (\ref{eq5}).

(1). Since $\beta-\frac{1}{2}=\frac{\beta}{p}+\frac{3}{2q}$, we know
$1+\frac{1}{p}=\frac{1}{p}+\frac{1}{p}+\frac{1}{r_{1}}$ with
$\frac{1}{r_{1}}=\frac{1}{2\beta}+\frac{3}{2\beta q}$. By Young's
inequality, we get
$$\|B(u,v)\|_{L^{p}((0,T),L^{q,\infty}(\mathbb{R}^{3}))}
\leq\|s^{-\frac{1}{2\beta}-\frac{3}{2\beta
q}}\|_{L^{r_{1},\infty}}\|u\|_{L^{p,\infty}((0,T),L^{q,\infty}(\mathbb{R}^{3}))}
\|v\|_{L^{p,\infty}((0,T),L^{q,\infty}(\mathbb{R}^{3}))}.$$ Now we
compute the norm $\|s^{-\frac{1}{r_{1}}}\|_{L^{r_{1},\infty}(0,T)},$
where
$$\|f\|_{L^{r_{1},\infty}(0,T)}=\sup_{\lambda}\lambda\left|\left\{t\in(0,T), |f(t)|>\lambda\right\}\right|^{1/p}.$$
If $s\in (0,T)$ and $T<\frac{1}{\lambda^{r_{1}}},$ then
$$\lambda\left|\left\{t\in(0,T),s^{-\frac{1}{r_{1}}}>\lambda\right\}\right|^{1/r_{1}}\leq\lambda T^{1/r_{1}}\leq 1.$$
If $T>\frac{1}{\lambda^{r_{1}}}$, then
$T^{-\frac{1}{r_{1}}}<\lambda$. For $s^{-\frac{1}{r_{1}}}>\lambda$,
we can find a $s_{0}$ such that $s_{0}=\frac{1}{\lambda^{r_{1}}}.$
When $0<s<s_{0}$,
$s^{-\frac{1}{r_{1}}}>s_{0}^{-\frac{1}{r_{1}}}=\lambda$, then we
have
$$\lambda\left|\left\{s\in(0,T): s^{-\frac{1}{r_{1}}}>\lambda\right\}\right|^{1/r_{1}}\leq\lambda s_{0}^{1/r_{1}}=1.$$
Therefore we get $s^{-\frac{1}{r_{1}}}\in L^{r_{1},\infty}((0,T))$
and $\|s^{-\frac{1}{r_{1}}}\|_{L^{r_{1},\infty}(o,T)}\leq 1.$

 (2). It follows from
(\ref{eq5}) that
$$\|B(u,v)(t)\|_{L^{q,\infty}(\mathbb{R}^{3})}
\leq\|u\|_{L^{\infty}((0,T),L^{q,\infty}(\mathbb{R}^{3}))}\int_{0}^{t}(t-s)^{-\frac{1}{2\beta}-\frac{3}{2\beta
q}}\|v(s)\|_{L^{q,\infty}(\mathbb{R}^{3})}ds.$$ Then we obtain
$$\|B(u,v)(t)\|_{L^{\infty}((0,T),L^{q,\infty}(\mathbb{R}^{3}))}
\leq\|u\|_{L^{\infty}((0,T),L^{q,\infty}(\mathbb{R}^{3}))}\sup_{0<t<T}\left|\int_{0}^{t}(t-s)^{-\frac{1}{2\beta}-\frac{3}{2\beta
q}}\|v(s)\|_{L^{q,\infty}(\mathbb{R}^{3})}ds\right|.$$ By
H\"{o}lder's inequality with
$1=\frac{1}{p}+\frac{1}{2\beta}+\frac{3}{2\beta q}$, we have
$$\|B(u,v)(t)\|_{L^{\infty}((0,T),L^{q,\infty}(\mathbb{R}^{3}))}
\leq
C\|u\|_{L^{\infty}((0,T),L^{q,\infty}(\mathbb{R}^{3}))}\|v\|_{L^{p,\infty}((0,T),L^{q,\infty}(\mathbb{R}^{3}))}.$$

(3). By (\ref{eq5}), we get
\begin{eqnarray*}
\|B(u,v)\|_{L^{\infty}((0,T),L^{q,\infty}(\mathbb{R}^{3}))}&\leq&\|u\|_{L^{\infty}((0,T),L^{q,\infty}(\mathbb{R}^{3}))}
\|v\|_{L^{\infty}((0,T),L^{q,\infty}(\mathbb{R}^{3}))}
\int_{0}^{t}(t-s)^{-(\frac{1}{2\beta}+\frac{3}{2\beta q})}ds\\
&\leq&CT^{1/p}\|u\|_{L^{\infty}((0,T),L^{q,\infty}(\mathbb{R}^{3}))}\|v\|_{L^{\infty}((0,T),L^{q,\infty}(\mathbb{R}^{3}))}.
\end{eqnarray*}

(4). (\ref{eq5}) and Young's inequality with
$1+\frac{1}{p}=\frac{1}{p}+\frac{1}{p}+(\frac{1}{2\beta}+\frac{3}{2\beta
q})$ imply that
\begin{eqnarray*}
\|B(u,v)(t)\|_{L^{q,\infty}(\mathbb{R}^{3})}&\leq&\int_{0}^{t}(t-s)^{-(\frac{1}{2\beta}+\frac{3}{2\beta
q})}\|u(s)\|_{L^{q,\infty}(\mathbb{R}^{3})}\|v(s)\|_{L^{q,\infty}(\mathbb{R}^{3})}ds\\
&\leq&\|u\|_{L^{\infty}((0,T),L^{q,\infty}(\mathbb{R}^{3}))}\left(\int_{0}^{t}(t-s)^{-(\frac{1}{2\beta}+\frac{3}{2\beta
q})}\|v(s)\|_{L^{q,\infty}(\mathbb{R}^{3})}ds\right).
\end{eqnarray*}
Since $\|f\ast g\|_{p,\infty}\leq\|f\|_{p,\infty}\|g\|_{1}$, we have
\begin{eqnarray*}
\|B(u,v)\|_{L^{p,\infty}((0,T),L^{q,\infty}(\mathbb{R}^{3}))}
&\leq&\|u\|_{L^{\infty}((0,T),L^{q,\infty}(\mathbb{R}^{3}))}\|v\|_{L^{p,\infty}((0,T),L^{q,\infty}(\mathbb{R}^{3}))}
\sup_{0<t<T}\int_{0}^{t}s^{-1+\frac{1}{p}}ds\\
&\leq&CT^{1/p}\|u\|_{L^{\infty}((0,T),L^{q,\infty}(\mathbb{R}^{3}))}\|v\|_{L^{p,\infty}((0,T),L^{q,\infty}(\mathbb{R}^{3}))}.
\end{eqnarray*}
This completes the proof of Lemma \ref{le1}.
\end{proof}

We need the following local existence of solution to  equations
(\ref{eq1e}) with initial data in Lorentz spaces.

\begin{lemma}\label{le2}
Let $\frac{1}{2}<\beta\leq 1,$ $\frac{3}{2\beta-1}<q\leq\infty$,
$\frac{2\beta}{2\beta-1}<p\leq\infty$ and $u_{0}\in
L^{q,\infty}(\mathbb{R}^{3})$. For $T>0$ such that
$4T^{1/p}\|u_{0}\|_{L^{q,\infty}(\mathbb{R}^{3})}<1$, there exists a
mild solution $u\in L^{\infty}((0,T),L^{q,\infty}(\mathbb{R}^{3}))$
to equations (\ref{eq1e}), which is unique in the ball centered at
$0$, of radius $2\|u_{0}\|_{L^{q,\infty}(\mathbb{R}^{3})}$.
\end{lemma}
\begin{proof}
We construct $\{e_{n}\}$ as follows:
\begin{equation} \label{eq:3}
\left\{ \begin{aligned}
         e_{n+1}&=e_{0}-B(e_{n},e_{n}),\\
                  e_{0}&=e^{-t(-\Delta)^{\beta}}u_{0}.
                          \end{aligned} \right.
                          \end{equation}
We claim that
$\|e_{n}\|_{L^{\infty}((0,T),L^{q,\infty}(\mathbb{R}^{3}))}\leq2\|u_{0}\|_{L^{q,\infty}(\mathbb{R}^{3})}$.
For $n=0$, by Young's inequality, we have
\begin{eqnarray*}
\|e_{0}\|_{L^{\infty}((0,T),L^{q,\infty}(\mathbb{R}^{3}))}&=&\|e^{-t(-\Delta)^{\beta}}u_{0}\|_{L^{\infty}((0,T),L^{q,\infty}(\mathbb{R}^{3}))}\\
&=&\sup_{t\in(0,T)}\|e^{-t(-\Delta)^{\beta}}u_{0}\|_{L^{q,\infty}(\mathbb{R}^{3})}\\
&\leq&\|e^{-t(-\Delta)^{\beta}}\|_{L^{1}(\mathbb{R}^{3})}\|u_{0}\|_{L^{q,\infty}(\mathbb{R}^{3})}\\
&\leq&\left(\int_{\mathbb{R}^{3}}\frac{1}{t^{\frac{3}{2\beta}}}\frac{1}{(1+\frac{|x-y|}{t^{1/2\beta}})^{3}}dy\right)\|u_{0}\|_{L^{q,\infty}(\mathbb{R}^{3})}\\
&\leq&2\|u_{0}\|_{L^{q,\infty}(\mathbb{R}^{3})}.
\end{eqnarray*}
Assume that the estimate is true for some $n\in\mathbb{N}$. For
$n+1$, we get
\begin{eqnarray*}
\|e_{n+1}\|_{L^{\infty}((0,T),L^{q,\infty}(\mathbb{R}^{3}))}
&\leq&\|e_{0}\|_{L^{\infty}((0,T),L^{q,\infty}(\mathbb{R}^{3}))}+\|B(e_{n},e_{n})\|_{L^{\infty}((0,T),L^{q,\infty}(\mathbb{R}^{3}))}\\
&\leq&\|e_{0}\|_{L^{\infty}((0,T),L^{q,\infty}(\mathbb{R}^{3}))}+T^{1/p}\|e_{n}\|^{2}_{L^{\infty}((0,T),L^{q,\infty}(\mathbb{R}^{3}))}\\
&\leq&\|u_{0}\|_{L^{q,\infty}(\mathbb{R}^{3})}+4T^{1/p}\|u_{0}\|^{2}_{L^{q,\infty}(\mathbb{R}^{3})}\\
&\leq&2\|u_{0}\|_{L^{q,\infty}(\mathbb{R}^{3})}.
\end{eqnarray*}
This tells us
\begin{eqnarray*}
&&\|e_{n+1}-e_{n}\|_{L^{\infty}((0,T),L^{q,\infty}(\mathbb{R}^{3}))}\\
&=&\|B(e_{n},e_{n})-B(e_{n-1},e_{n-1})\|_{L^{\infty}((0,T),L^{q,\infty}(\mathbb{R}^{3}))}\\
&\leq&T^{1/p}\|e_{n}-e_{n-1}\|_{L^{\infty}((0,T),L^{q,\infty}(\mathbb{R}^{3}))}
\left(|e_{n}\|_{L^{\infty}((0,T),L^{q,\infty}(\mathbb{R}^{3}))}+\|e_{n-1}\|_{L^{\infty}((0,T),L^{q,\infty}(\mathbb{R}^{3}))}\right)\\
&\leq&4T^{1/p}\|u_{0}\|_{L^{q,\infty}(\mathbb{R}^{3})}\|e_{n}-e_{n-1}\|_{L^{\infty}((0,T),L^{q,\infty}(\mathbb{R}^{3}))}.
\end{eqnarray*}
Since $4T^{1/p}\|u_{0}\|_{L^{q,\infty}(\mathbb{R}^{3})}<1$, the
Picard contraction principle guarantees this lemma.
\end{proof}

 \vspace{0.1in}
\section{Main Results}

In this section, we state and prove our main results. First, we need
the following proposition which  generalizes the case $\beta=1$
established by Lemari$\acute{e}$-Rieusset and Prioux  \cite{LRP}.

\begin{proposition}\label{pro2}
Let $T>0$, $\frac{1}{2}<\beta<1$ and $u$, $v$ be two mild solutions
to the equations (\ref{eq1e}) belonging to the space
$\widetilde{L}^{p,\infty}((0,T),L^{q,\infty}(\mathbb{R}^{3}))$ with
$\frac{2\beta}{2\beta-1}<p<\infty$ and $\frac{3}{2\beta-1}<q<\infty$
such that $\beta-\frac{1}{2}=\frac{\beta}{p}+\frac{3}{2q}$. Assume
that there exists $\theta\in(0,T)$ such that $u(\theta)=v(\theta)$.
Then $u$ and $v$ are equal for $t\in(\theta,T]$.
\end{proposition}
\begin{proof}
Let $t_{0}>0$ and $\lambda>0$. We can split $u$ and $v$ into:
$$u=u_{\lambda}+u'_{\lambda}\quad\text{and}\quad v=v_{\lambda}+v'_{\lambda}$$
where
$u_{\lambda}=u\chi_{\{t:\|u(t)\|_{L^{q,\infty}(\mathbb{R}^{3})}>\lambda\}}$
and
$v_{\lambda}=v\chi_{\{t:\|v(t)\|_{{q,\infty}(\mathbb{R}^{3})}>\lambda\}}$.
By construction and the definition of the Lorentz spaces (see
Proposition \ref{pro1}) we have
$$\|u'_{\lambda}\|_{L^{\infty}((\theta,\theta+t_{0}),L^{q,\infty}(\mathbb{R}^{3}))}\leq\lambda
,\qquad\|u_{\lambda}\|_{L^{p,\infty}((\theta,\theta+t_{0}),L^{q,\infty}(\mathbb{R}^{3}))}\leq
C(\lambda)$$ and the same estimates hold true for $v'_{\lambda}$ and
$v_{\lambda}$. Then, we compute by Lemma \ref{le1}(1),
\begin{eqnarray*}
&&\|u-v\|_{L^{p,\infty}((\theta,\theta+t_{0}),L^{q,\infty}(\mathbb{R}^{3}))}\\
&\leq&\|B(u,u)-B(v,v)\|_{L^{p,\infty}((\theta,\theta+t_{0}),L^{q,\infty}(\mathbb{R}^{3}))}\\
&\leq&C_{0}\|B(u-v,u)\|_{L^{p,\infty}((\theta,\theta+t_{0}),L^{q,\infty}(\mathbb{R}^{3}))}
+C_{0}\|B(v,u-v)\|_{L^{p,\infty}((\theta,\theta+t_{0}),L^{q,\infty}(\mathbb{R}^{3}))}\\
&\leq&C\|u-v\|_{L^{p,\infty}((\theta,\theta+t_{0}),L^{q,\infty}(\mathbb{R}^{3}))}
\left(\|u\|_{L^{p,\infty}((\theta,\theta+t_{0}),L^{q,\infty}(\mathbb{R}^{3}))}+
\|v\|_{L^{p,\infty}((\theta,\theta+t_{0}),L^{q,\infty}(\mathbb{R}^{3}))}\right).
\end{eqnarray*}
Since $u=u_{\lambda}+u'_{\lambda}$ with
$\|u_{\lambda}\|_{L^{p,\infty}((\theta,\theta+t_{0}),L^{q,\infty}(\mathbb{R}^{3}))}\leq
C(\lambda)$ and
$\|u'_{\lambda}\|_{L^{p,\infty}((\theta,\theta+t_{0}),L^{q,\infty}(\mathbb{R}^{3}))}\leq\lambda$,
 we get
$\|u\|_{L^{p,\infty}((\theta,\theta+t_{0}),L^{q,\infty}(\mathbb{R}^{3}))}\leq
C(\lambda)+\lambda t_{0}^{1/p}$. The same estimate holds for $v$.
Hence we can obtain
$$\|u-v\|_{L^{p,\infty}((\theta,\theta+t_{0}),L^{q,\infty}(\mathbb{R}^{3}))}
\leq C_{0}\left(2C(\lambda)+2\lambda
t_{0}^{1/p}\right)\|u-v\|_{L^{p,\infty}((\theta,\theta+t_{0}),L^{q,\infty}(\mathbb{R}^{3}))}.$$
We choose $\lambda>0$ large enough to guarantee
$2C_{0}C(\lambda)<1/4$ and choose $t_{0}>0$ small enough such that
$C_{0}t_{0}^{1/p}<1/4$. Thus there exists $\delta<1$ satisfies
$$\|u-v\|_{L^{p,\infty}((\theta,\theta+t_{0}),L^{q,\infty}(\mathbb{R}^{3}))}
\leq
\delta\|u-v\|_{L^{p,\infty}((\theta,\theta+t_{0}),L^{q,\infty}(\mathbb{R}^{3}))}.$$
So $u=v$ for $t\in(\theta,\theta+t_{0})$. For $T$, there exists $n$
such that $T<\theta+nt_{0}$. Thus $u=v$ for $t\in(\theta,T]$.
\end{proof}

\begin{lemma}\label{le3}
(\cite[Proposition 2.9]{LRP}) Let $T>0$ and $1\leq p,q\leq\infty$.
If $u$ satisfies
\begin{equation} \label{eq:2}
\left\{ \begin{aligned}
         \sup_{t\in(0,T)}t^{1/p}\|u(t)\|_{L^{q,\infty}(\mathbb{R}^{3})}&<\infty,\\
                  t^{1/p}\|u(t)\|_{L^{q,\infty}(\mathbb{R}^{3})}&\longrightarrow0,
                  (\hbox{as}\ t\rightarrow0),
                          \end{aligned} \right.
                          \end{equation}
then the function $u$ belongs to the space
$\widetilde{L}^{p,\infty}((0,T),L^{q,\infty}(\mathbb{R}^{3})).$
\end{lemma}

To  establish the equivalence between the mild and mollified
solution to the $(GNS)$ equations,  we need the following lemma.
\begin{lemma}\label{le5}
Let $v\in\mathcal{D}((0,T]\times \bar{B}(0,R))$ and $R>0$ such that
$\text{supp }v\subset(0,T]\times \bar{B}(0,R)$ where supp denotes
the supports of the function $v$ and $\bar{B}(0,R)$ the closed ball
of radius $R$ centered at $0$. Then, for $t\in (0,T]$ and
$y\in\mathbb{R}^{3}$ such that $|y|\geq\lambda R$ for $\lambda>1$,
we have for some constant $C>0$,
$$\left|B(u,u)(t,y)\right|\leq\frac{C}{(\lambda R)^{4}}\|u\|^{2}_{L^{2}((0,T]\times\mathbb{R}^{3})}.$$
\end{lemma}
\begin{proof}
Since $|y|\geq\lambda R>\lambda|z|$,
$|y-z|\geq(1-\frac{1}{\lambda})|y|$. Then, we can get
\begin{eqnarray*}
|B(u,u)(t,y)|&=&\left|\int_{0}^{t}e^{-(t-s)(-\Delta)^{\beta}}\mathbb{P}\nabla(u\otimes
u)ds\right|\\
&\leq&C\int_{0}^{t}\int_{\mathbb{R}^{3}}\frac{1}{((t-s)^{\frac{1}{2\beta}}+|z-y|)^{4}}|v(s,z)|^{2}dsdz\\
&\leq&C\int_{0}^{t}\int_{\mathbb{R}^{3}}\frac{1}{|z-y|^{4}}|v(s,z)|^{2}dsdz\\
&\leq&C\int_{0}^{t}\int_{B(0,R)}\frac{1}{|z-y|^{4}}|v(s,z)|^{2}dsdz\\
&\leq&C\frac{1}{|y|^{4}}\int_{0}^{t}\int_{B(0,R)}|v(s,z)|^{2}dsdz\\
&=&C\frac{1}{|y|^{4}}\|v\|^{2}_{L^{2}((0,T)\times
\mathbb{R}^{3})}\lesssim\frac{1}{|\lambda
R|^{4}}\|v\|^{2}_{L^{2}((0,T)\times \mathbb{R}^{3})}.
\end{eqnarray*}
\end{proof}

\begin{theorem}\label{th2}
Let $\alpha>0$, $\max\left\{\frac{1}{2},\alpha\right\}<\beta\leq1$
with $\alpha+\beta-1\geq0$ and let
$u_{0}\in\overline{\mathcal{D}(\mathbb{R}^{3})}^{Q^{\beta,-1}_{\alpha,loc}(\mathbb{R}^{3})}$
such that $\nabla\cdot u_{0}=0$ and $T>0$ small enough to ensure
$\|e^{-t(-\Delta)^{\beta}}u_{0}\|_{X^{\beta}_{\alpha;T}(\mathbb{R}^{3})}<\frac{1}{4C}$.
Then there exists a mild solution $u\in
\overline{\mathcal{D}((0,T)\times\mathbb{R}^{3})}^{X^{\beta}_{\alpha;T}(\mathbb{R}^{3})}$
to  equations (\ref{eq1e}).
\end{theorem}
\begin{proof}
We construct $\{v_{n}\}_{n\in\mathbb{N}}$ by

\begin{equation} \label{eq8}
\left\{ \begin{aligned}
         v_{n}&=v_{0}-B(v_{n-1},v_{n-1}),\text{ for }n\geq1, \\
                 v_{0}&=e^{-t(-\Delta)^{\beta}}u_{0}.
                          \end{aligned} \right.
                          \end{equation}
For $n=0$. By assumption, if
$u_{0}\in\overline{\mathcal{D}(\mathbb{R}^{3})}^{Q^{\beta,-1}_{\alpha,loc}(\mathbb{R}^{3})}$,
 there exists a sequence
$u^{m}_{0}\in\mathcal{D}(\mathbb{R}^{3})$ such that
$\|u_{0}-u^{m}_{0}\|_{Q^{\beta,-1}_{\alpha,loc}(\mathbb{R}^{3})}\longrightarrow0$
as $m\rightarrow\infty$. From the definition of
$Q^{\beta,-1}_{\alpha,loc}(\mathbb{R}^{3}),$  if $f\in
Q^{\beta,-1}_{\alpha,loc}(\mathbb{R}^{3})$,
$$\sup_{0<t^{2\beta}<T}\sup_{x_{0}\in\mathbb{R}^{3}}t^{2\alpha-3+2\beta-2}\int_{0}^{t^{2\beta}}
\int_{|x-x_{0}|<t}|e^{-s(-\Delta)^{\beta}}f(x)|^{2}\frac{dsdx}{s^{\alpha/\beta}}<\infty.$$
Hence,  as $m\rightarrow\infty$,
$$\sup_{0<t^{2\beta}<T}\sup_{x_{0}\in\mathbb{R}^{3}}t^{2\alpha-3+2\beta-2}\int_{0}^{t^{2\beta}}
\int_{|x-x_{0}|<t}|e^{-s(-\Delta)^{\beta}}(u_{0}-u^{m}_{0})(x)|^{2}\frac{dsdx}{s^{\alpha/\beta}}\longrightarrow0.$$
From the embedding:
$Q^{\beta,-1}_{\alpha;T}(\mathbb{R}^{3})\hookrightarrow
\dot{B}^{1-2\beta}_{\infty,\infty}(\mathbb{R}^{3})$ (see
\cite[Theorem 4.6]{Li Zhai}), we obtain
$$t^{\frac{2\beta-1}{2\beta}}\|e^{-t(-\Delta)^{\beta}}f\|_{L^{\infty}(\mathbb{R}^{3})}\leq C\|f\|_{Q^{\beta,-1}_{\alpha;T}(\mathbb{R}^{3})}.$$
By the definition of $X^{\beta}_{\alpha;T}(\mathbb{R}^{3}),$ we get
$\|e^{-t(-\Delta)^{\beta}}f\|_{X^{\beta}_{\alpha;T}(\mathbb{R}^{3})}\leq
C\|f\|_{Q^{\beta,-1}_{\alpha;T}(\mathbb{R}^{3})}.$ Then we have
$$\|e^{-t(-\Delta)^{\beta}}(u_{0}-u_{0}^{m})\|_{X^{\beta,-1}_{\alpha;T}(\mathbb{R}^{3})}\leq
C\|u_{0}-u_{0}^{m}\|_{Q^{\beta,-1}_{\alpha;T}(\mathbb{R}^{3})}$$ and
$\|e^{-t(-\Delta)^{\beta}}(u_{0}-u_{0}^{m})\|_{X^{\beta,-1}_{\alpha;T}(\mathbb{R}^{3})}\longrightarrow0\
\ \hbox{as}\ \  m\rightarrow0.$ So
$e^{-t(-\Delta)^{\beta}}u_{0}=\lim_{m\rightarrow\infty}e^{-t(-\Delta)^{\beta}}u_{0}^{m}$
in $X^{\beta}_{\alpha;T}(\mathbb{R}^{3})$. It follows from
$e^{-t(-\Delta)^{\beta}}u_{0}^{m}\in \mathcal{D}((0,T)\times
\mathbb{R}^{3})$ that
$$v_{0}=e^{-t(-\Delta)^{\beta}}u_{0}\in\overline{\mathcal{D}((0,T)\times\mathbb{R}^{3})}^{X^{\beta}_{\alpha;T}(\mathbb{R}^{3})}.$$
Let us assume
$v_{n-1}\in\overline{\mathcal{D}((0,T)\times\mathbb{R}^{3})}^{X^{\beta}_{\alpha;T}(\mathbb{R}^{3})}$.
For $v_{n}$, since
$v_{n}=e^{-t(-\Delta)^{\beta}}u_{0}-B(v_{n-1},v_{n-1})$,
$$u_{0}\in\overline{\mathcal{D}(\mathbb{R}^{3})}^{Q^{\beta,-1}_{\alpha;T}(\mathbb{R}^{3})}
\Longrightarrow
v_{0}=e^{-t(-\Delta)^{\beta}}u_{0}\in\overline{\mathcal{D}((0,T)\times\mathbb{R}^{3})}^{X^{\beta}_{\alpha;T}(\mathbb{R}^{3})}.$$
We only need to prove
$B(v_{n-1},v_{n-1})\in\overline{\mathcal{D}((0,T)\times\mathbb{R}^{3})}^{X^{\beta}_{\alpha;T}(\mathbb{R}^{3})}.$

By induction, there exists a sequence
$v^{m}_{n-1}\in\mathcal{D}((0,T)\times\mathbb{R}^{3})$ such that
$$\|v_{n-1}-v^{m}_{n-1}\|_{X^{\beta}_{\alpha;T}(\mathbb{R}^{3})}\longrightarrow0,
\ \ \hbox{as} \ \ m\rightarrow\infty.$$ Since $v^{m}_{n-1}$ is
compact supported in time and space, we
 have $B(v^{m}_{n-1},v^{m}_{n-1})\in
C^{\infty}((0,T]\times\mathbb{R}^{3})$ and is of compact support in
time. Let $\left\{\varphi_{m}\right\}_{m\in \mathbb{N}}$ be a
sequence of functions in $\mathcal{D}(\mathbb{R}^{3})$ such that for
each $m\in\mathbb{N}$, $\|\varphi_{m}\|_{\infty}=1$. Assume
supp$\varphi_{m}\subset\bar{B}(0,\lambda_{m}R_{m}+1)$ and
$\varphi_{m}(x)=1$ if $x\in B(0,\lambda_{m}R_{m})$ where $R_{m}>0$
is such that $\text{supp }v^{m}_{n-1}\subset(0,T]\times B(0,R_{m})$
and $\lambda_{m}>m\|v^{m}_{n-1}\|^{1/2}_{L^{2}((0,T)\times
\mathbb{R}^{3})}$. We denote
$B^{m}(v_{n-1},v_{n-1})=\varphi_{m}\times
B(v^{m}_{n-1},v^{m}_{n-1})$ and get
\begin{eqnarray*}
&&\|B(v_{n-1},v_{n-1})-B^{m}(v^{m}_{n-1},v^{m}_{n-1})\|_{X^{\beta}_{\alpha;T}(\mathbb{R}^{3})}\\
&\leq&\|B(v_{n-1},v_{n-1})-B(v^{m}_{n-1},v^{m}_{n-1})\|_{X^{\beta}_{\alpha;T}(\mathbb{R}^{3})}
+\|(1-\varphi_{m})B(v^{m}_{n-1},v^{m}_{n-1})\|_{X^{\beta}_{\alpha;T}(\mathbb{R}^{3})}\\
&\leq&
C\|v_{n-1}-v_{n-1}^{m}\|_{X^{\beta}_{\alpha;T}}\left[\|v_{n-1}\|_{X^{\beta}_{\alpha;T}(\mathbb{R}^{3})}
+\|v_{n-1}^{m}\|_{X^{\beta}_{\alpha;T}(\mathbb{R}^{3})}\right]\\
&&+\|(1-\varphi_{m})B(v^{m}_{n-1},v^{m}_{n-1})\|_{X^{\beta}_{\alpha;T}(\mathbb{R}^{3})}.
\end{eqnarray*}
Since $\varphi_{m}$ is supported on $\bar{B}(0,\lambda_{m}R_{m}+1)$
and $\varphi_{m}=1$ on $B(0,\lambda_{m}R_{m})$, $(1-\varphi_{m}(y))$
is supported on $\bar{B}^{c}(0,\lambda_{m}R_{m})=\{y:
|y|>\lambda_{m}R_{m}\}$. Then, we obtain
\begin{eqnarray*}
&&\|(1-\varphi_{m})B(v^{m}_{n-1},v^{m}_{n-1})\|_{X^{\beta}_{\alpha; T}(\mathbb{R}^{3})}\\
&\leq&\sup_{t\in(0,T)}t^{1-\frac{1}{2\beta}}\|(1-\varphi_{m})B(v^{m}_{n-1},v^{m}_{n-1})\|_{L^{\infty}(\mathbb{R}^{3})}\\
&&+\sup_{t^{2\beta}\in(0,T)}\sup_{x_{0}\in\mathbb{R}^{3}}\left(t^{2\alpha-3+2\beta+2\beta-2}
\int^{t^{2\beta}}_{0}\int_{|y-x_{0}|<t}|(1-\varphi_{m})B(v^{m}_{n-1},v^{m}_{n-1})(s,y)|^{2}\frac{dsdy}{s^{\alpha/\beta}}\right)^{1/2}\\
&\leq&\sup_{t\in(0,T)}t^{1-\frac{1}{2\beta}}\frac{1}{(\lambda_{m}R_{m})^{4}}\|v^{m}_{n-1}\|^{2}_{L^{2}((0,T]\times\mathbb{R}^{3})}\\
&&+\sup_{t^{2\beta}\in(0,T)}\sup_{x_{0}\in\mathbb{R}^{3}}\|v^{m}_{n-1}\|^{2}_{L^{2}((0,T]\times\mathbb{R}^{3})}
\frac{1}{(\lambda_{m}R_{m})^{4}}\left(t^{2\alpha-3+2\beta+2\beta-2}\int^{t^{2\beta}}_{0}\int_{|y-x_{0}|<t}\frac{dsdy}{s^{\alpha/\beta}}\right)^{1/2}\\
&\leq&
CT^{1-\frac{1}{2\beta}}\frac{1}{(\lambda_{m}R_{m})^{4}}\|v^{m}_{n-1}\|^{2}_{L^{2}((0,T]\times\mathbb{R}^{3})}+
\frac{\|v^{m}_{n-1}\|^{2}_{L^{2}((0,T]\times\mathbb{R}^{3})}}{(\lambda_{m}R_{m})^{4}}\sup_{t^{2\beta}\in(0,T)}
(t^{2\alpha-3+2\beta-2}t^{3+2\beta(1-\alpha/\beta)})^{1/2}\\
&\leq&
CT^{1-\frac{1}{2\beta}}\frac{1}{(\lambda_{m}R_{m})^{4}}\|v^{m}_{n-1}\|^{2}_{L^{2}((0,T]\times\mathbb{R}^{3})}\\
&\leq&
CT^{1-\frac{1}{2\beta}}\frac{1}{(mR_{m})^{4}}\rightarrow0\quad(\hbox{as}\
m\rightarrow\infty).
\end{eqnarray*}
Thus,
$\|B(v_{n-1},v_{n-1})-B^{m}(v_{n-1},v_{n-1})\|_{X^{\beta}_{\alpha;T}(\mathbb{R}^{3})}\longrightarrow0$
as $m\rightarrow0$, that is,
$v_{n}\in\overline{\mathcal{D}((0,T]\times\mathbb{R}^{3})}^{X^{\beta}_{\alpha;T}(\mathbb{R}^{3})}.$

Next we prove that $v_{n}$ have a limit in
$X^{\beta}_{\alpha;T}(\mathbb{R}^{3})$. We prove
$\|v_{n}\|_{X^{\beta}_{\alpha;T}(\mathbb{R}^{3})}\leq2\|e^{-t(-\Delta)^{\beta}}u_{0}\|_{X^{\beta}_{\alpha;T}(\mathbb{R}^{3})}.$
It follows from  $v_{0}=e^{-t(-\Delta)^{\beta}}u_{0}$ that
$$\|v_{0}\|_{X^{\beta}_{\alpha;T}(\mathbb{R}^{3})}
=\|e^{-t(-\Delta)^{\beta}}u_{0}\|_{X_{\alpha;T}^{\beta}(\mathbb{R}^{3})}\leq2\|e^{-t(-\Delta)^{\beta}}u_{0}\|_{X_{\alpha;T}^{\beta}(\mathbb{R}^{3})}.$$

We assume that for $n\in\mathbb{N}$,
$\|v_{n}\|_{X^{\beta}_{\alpha;T}(\mathbb{R}^{3})}\leq2\|e^{-t(-\Delta)^{\beta}}u_{0}\|_{X_{\alpha;T}^{\beta}(\mathbb{R}^{3})}.$
Then we get
\begin{eqnarray*}
\|v_{n+1}\|_{X^{\beta}_{\alpha;T}(\mathbb{R}^{3})}
&\leq&\|v_{0}\|_{X^{\beta}_{\alpha;T}(\mathbb{R}^{3})}+\|B(v_{n},v_{n})\|_{X^{\beta}_{\alpha;T}(\mathbb{R}^{3})}\\
&\leq&\|e^{-t(-\Delta)^{\beta}}u_{0}\|_{X_{\alpha;T}^{\beta}(\mathbb{R}^{3})}+C\|v_{n}\|^{2}_{X^{\beta}_{\alpha;T}(\mathbb{R}^{3})}\\
&\leq&\|e^{-t(-\Delta)^{\beta}}u_{0}\|_{X_{\alpha;T}^{\beta}(\mathbb{R}^{3})}+4C\|e^{-t(-\Delta)^{\beta}}u_{0}\|^{2}_{X_{\alpha;T}^{\beta}(\mathbb{R}^{3})}.
\end{eqnarray*}
It follows from
$\|e^{-t(-\Delta)^{\beta}}u_{0}\|_{X^{\beta}_{\alpha;T}(\mathbb{R}^{3})}<\frac{1}{4C}$
that
 $\|v_{n+1}\|_{X^{\beta}_{\alpha;T}(\mathbb{R}^{3})}\leq
2\|e^{-t(-\Delta)^{\beta}}u_{0}\|_{X_{\alpha;T}^{\beta}(\mathbb{R}^{3})}$.
Moreover,
\begin{eqnarray*}
\|v_{n}-v_{n-1}\|_{X^{\beta}_{\alpha;T}(\mathbb{R}^{3})}&\leq&\|B(v_{n-1},v_{n-1})-B(v_{n-2},v_{n-2})\|_{X^{\beta}_{\alpha;T}(\mathbb{R}^{3})}\\
&\leq&C\|v_{n-1}-v_{n-2}\|_{X^{\beta}_{\alpha;T}(\mathbb{R}^{3})}(\|v_{n-1}\|_{X^{\beta}_{\alpha;T}(\mathbb{R}^{3})}
+\|v_{n-2}\|_{X^{\beta}_{\alpha;T}(\mathbb{R}^{3})})\\
&\leq&4C\|e^{-t(-\Delta)^{\beta}}u_{0}\|_{X_{\alpha;T}^{\beta}(\mathbb{R}^{3})}\|v_{n-1}-v_{n-2}\|_{X^{\beta}_{\alpha;T}(\mathbb{R}^{3})}\\
&\leq&(4C\|e^{-t(-\Delta)^{\beta}}u_{0}\|_{X_{\alpha;T}^{\beta}(\mathbb{R}^{3})})^{n}\|v_{1}-v_{0}\|_{X^{\beta}_{\alpha;T}(\mathbb{R}^{3})}.
\end{eqnarray*}
Since
$4C\|e^{-t(-\Delta)^{\beta}}u_{0}\|_{X_{\alpha;T}^{\beta}(\mathbb{R}^{3})}<1,$
the Picard contraction principle   implies  the desired.
\end{proof}

\begin{theorem}\label{th3}
Let $\alpha>0$, $\max\left\{\frac{1}{2},\alpha\right\}<\beta\leq1$
with $\alpha+\beta-1\geq0$ and let
$u_{0}\in\overline{\mathcal{D}(\mathbb{R}^{3})}^{Q^{\beta,-1}_{\alpha,loc}(\mathbb{R}^{3})}$
such that $\nabla\cdot u_{0}=0$ and $T>0$ is small enough to ensure
$\|e^{-t(-\Delta)^{\beta}}u_{0}\|_{X^{\beta}_{\alpha;T}(\mathbb{R}^{3})}<\frac{1}{4C}$.
Then for $\varepsilon>0$, there exists a solution
$u_{\varepsilon}\in\overline{\mathcal{D}((0,T]\times\mathbb{R}^{3})}^{X^{\beta}_{\alpha;T}(\mathbb{R}^{3})}$
to the mollified generalized Navier-Stokes equations (\ref{eq3}).
\end{theorem}

\begin{proof}
We only need to prove
$\|f\ast\omega_{\varepsilon}\|_{X^{\beta}_{\alpha;T}(\mathbb{R}^{3})}\leq\|f\|_{X^{\beta}_{\alpha;T}(\mathbb{R}^{3})}.$
In fact,  we have $\|\omega_{\varepsilon}\ast
f\|_{L^{\infty}(\mathbb{R}^{3})}\leq\|\omega_{\varepsilon}\|_{L^{1}(\mathbb{R}^{3})}\|f\|_{L^{\infty}(\mathbb{R}^{3})}$
and
\begin{eqnarray*}
&&\left(r^{2\alpha-3+2\beta-2}\int_{0}^{r^{2\beta}}\int_{|x-x_{0}|<r}
\left|f\ast\omega_{\varepsilon}(t,x)\right|^{2}\frac{dtdx}{t^{\alpha/\beta}}\right)^{1/2}\\
&\leq&\left(r^{2\alpha-3+2\beta-2}\int_{0}^{r^{2\beta}}\int_{\mathbb{R}^{3}}
\left|\int_{\mathbb{R}^{3}}\chi_{B(x_{0},r)}f(t,x-y)\omega_{\varepsilon}(y)dy\right|^{2}\frac{dtdx}{t^{\alpha/\beta}}\right)^{1/2}\\
&\leq&\int_{\mathbb{R}^{3}}|\omega_{\varepsilon}(y)|\left(r^{2\alpha-3+2\beta-2}\int_{0}^{r^{2\beta}}\int_{\mathbb{R}^{3}}
\left|f(t,x-y)\right|^{2}\chi_{B(x_{0},r)}\frac{dtdx}{t^{\alpha/\beta}}\right)^{1/2}dy\\
&\leq&\int_{\mathbb{R}^{3}}|\omega_{\varepsilon}(y)|\left(r^{2\alpha-3+2\beta-2}\int_{0}^{r^{2\beta}}\int_{|x_{1}-(x_{0}-y)|<r}
\left|f(t,x_{1})\right|^{2}\frac{dtdx_{1}}{t^{\alpha/\beta}}\right)^{1/2}dy\\
&\leq&\int_{\mathbb{R}^{3}}|\omega_{\varepsilon}(y)|dy\sup_{z\in\mathbb{R}^{3}}\sup_{r^{2\beta}\in(0,T]}
\left(r^{2\alpha-3+2\beta-2}\int_{0}^{r^{2\beta}}\int_{|x_{1}-z|<r}
\left|f(t,x_{1})\right|^{2}\frac{dtdx_{1}}{t^{\alpha/\beta}}\right)^{1/2}\\
&\leq&
\|\omega_{\varepsilon}\|_{L^{1}(\mathbb{R}^{3})}\|f\|_{X^{\beta}_{\alpha;T}(\mathbb{R}^{3})}.
\end{eqnarray*}
Similar to the proof of  Theorem \ref{th2}, we can complete the
proof.

\end{proof}
\begin{theorem}\label{th4}
For $\alpha>0$, $\max\left\{\frac{1}{2},\alpha\right\}<\beta\leq1$
with $\alpha+\beta-1\geq0,$ let
$u_{0}\in\overline{\mathcal{D}(\mathbb{R}^{3})}^{Q^{\beta,-1}_{\alpha;loc}(\mathbb{R}^{3})}$
and  $T>0$ be given in Theorem \ref{th2}. Then
  the sequence of  solutions $\{u_{\varepsilon}\}_{\varepsilon>0}$ to the
mollified equations (\ref{eq3}) obtained by Theorem \ref{th3}
converges strongly, as $\varepsilon$ tends to $0,$  to the mild
solution $u$ to equations (\ref{eq1e})  obtained by Picard
contraction principle, of Theorem \ref{th2}.
\end{theorem}

\begin{proof}
For the bilinear form $B(u,v)$,  we have
\begin{eqnarray*}
u-u_{\varepsilon}&=&B(u,u)-B_{\varepsilon}(u_{\varepsilon},u_{\varepsilon})\\
&=&B(u,u)-B(u_{\varepsilon}\ast\omega_{\varepsilon},u_{\varepsilon})\\
&=&B(u,u-u_{\varepsilon})+B(u-(u\ast\omega_{\varepsilon}),u_{\varepsilon})+B((u-u_{\varepsilon})\ast\omega_{\varepsilon},u_{\varepsilon})
\end{eqnarray*}
and
\begin{eqnarray*}
&&\|u-u_{\varepsilon}\|_{X^{\beta}_{\alpha;T}(\mathbb{R}^{3})}\\
&\leq&
C\|u\|_{X^{\beta}_{\alpha;T}(\mathbb{R}^{3})}\|u-u_{\varepsilon}\|_{X^{\beta}_{\alpha;T}(\mathbb{R}^{3})}
+C\|u-(u\ast\omega_{\varepsilon})\|_{X^{\beta}_{\alpha;T}(\mathbb{R}^{3})}
\|u_{\varepsilon}\|_{X^{\beta}_{\alpha;T}(\mathbb{R}^{3})}\\
&&+C\|(u-u_{\varepsilon})\ast\omega_{\varepsilon}\|_{X^{\beta}_{\alpha;T}(\mathbb{R}^{3})(\mathbb{R}^{3})}
\|u_{\varepsilon}\|_{X^{\beta}_{\alpha;T}(\mathbb{R}^{3})}\\
&:=&A_{1}+A_{2}+A_{3}
\end{eqnarray*}
where $A_{3}\leq
C\|\omega_{\varepsilon}\|_{L^{1}(\mathbb{R}^{3})}\|u-u_{\varepsilon}\|_{X^{\beta}_{\alpha;T}(\mathbb{R}^{3})}
\|u_{\varepsilon}\|_{X^{\beta}_{\alpha;T}(\mathbb{R}^{3})}.$ Hence
we have
\begin{eqnarray*}
&&\|u-u_{\varepsilon}\|_{X^{\beta}_{\alpha;T}(\mathbb{R}^{3})}\\
&\leq&2C\|u-u_{\varepsilon}\|_{X^{\beta}_{\alpha;T}(\mathbb{R}^{3})}\|u_{\varepsilon}\|_{X^{\beta}_{\alpha;T}(\mathbb{R}^{3})}
+2C\|u-(u\ast\omega_{\varepsilon})\|_{X^{\beta}_{\alpha;T}(\mathbb{R}^{3})}
\|u_{\varepsilon}\|_{X^{\beta}_{\alpha;T}(\mathbb{R}^{3})}\\
&\leq&4C\|e^{-t(-\Delta)^{\beta}}u_{0}\|_{X^{\beta}_{\alpha;T}(\mathbb{R}^{3})}\|u-u_{\varepsilon}\|_{X^{\beta}_{\alpha;T}(\mathbb{R}^{3})}
+2C\|e^{-t(-\Delta)^{\beta}}u_{0}\|_{X^{\beta}_{\alpha;T}(\mathbb{R}^{3})}\|u-(u\ast\omega_{\varepsilon})\|_{X^{\beta}_{\alpha;T}(\mathbb{R}^{3})}.
\end{eqnarray*}
This tells us
$$\|u-u_{\varepsilon}\|_{X^{\beta}_{\alpha;T}(\mathbb{R}^{3})}\leq\frac{2C\|e^{-t(-\Delta)^{\beta}}u_{0}\|_{X^{\beta}_{\alpha;T}(\mathbb{R}^{3})}}
{1-4C\|e^{-t(-\Delta)^{\beta}}u_{0}\|_{X^{\beta}_{\alpha;T}(\mathbb{R}^{3})}}\|u-(u\ast\omega_{\varepsilon})\|_{X^{\beta}_{\alpha;T}(\mathbb{R}^{3})}.$$
Since $\omega_{\varepsilon}\in\mathcal{D}(\mathbb{R}^{3})$,
$\omega_{\varepsilon}\ast
u\in\mathcal{D}(\mathbb{R}^{3}\times(0,T))$. Thus,  for
$u\in\overline{\mathcal{D}(\mathbb{R}^{3}\times(0,T))}^{X^{\beta}_{\alpha;T}(\mathbb{R}^{3})},$
$$\|u-(u\ast\omega_{\varepsilon})\|_{X_{\alpha;T}^{\beta}(\mathbb{R}^{3})}\longrightarrow0\quad\text{as }\varepsilon\rightarrow0.$$
\end{proof}

Now we recall a class of weak Besov spaces which can be found in
\cite{LRP}.
\begin{definition}\label{Def2}
Let $\alpha>0$, $1<q<\infty$. We denote by
$\widetilde{B}^{-\alpha,\infty}_{q}(\mathbb{R}^{3})$ the adherence
of functions in $L^{q}(\mathbb{R}^{3})$ for the norm of
$B^{-\alpha,\infty}_{q}(\mathbb{R}^{3})$ and by
$\widetilde{B}^{-\alpha,\infty}_{q,\infty}(\mathbb{R}^{3})$ for
functions in $L^{q,\infty}(\mathbb{R}^{3})$ for the norm of
$B^{-\alpha,\infty}_{q,\infty}(\mathbb{R}^{3})=B^{-\alpha,\infty}_{L^{q},\infty}(\mathbb{R}^{3})$,
that is,
$$\widetilde{B}^{-\alpha,\infty}_{q}(\mathbb{R}^{3})=\overline{L^{q}(\mathbb{R}^{3})}^{B^{-\alpha,\infty}_{q,\infty}(\mathbb{R}^{3})}
\quad\text{and}\quad\widetilde{B}^{-\alpha,\infty}_{q,\infty}(\mathbb{R}^{3})
=\overline{L^{q,\infty}(\mathbb{R}^{3})}^{B^{-\alpha,\infty}_{q,\infty}(\mathbb{R}^{3})}.$$

\end{definition}
\begin{lemma}\label{le4}
Let $\frac{1}{2}<\beta< 1$ and let $\alpha>0$ and $1<q<\infty$.
If\quad
$u\in\widetilde{B}^{-\alpha,\infty}_{q,\infty}(\mathbb{R}^{3})$,
then

\begin{equation} \label{eq7}
\left\{ \begin{aligned}
         \sup_{0<t<1}t^{\alpha/2\beta}\|e^{-t(-\Delta)^{\beta}}u(t)\|_{L^{q,\infty}(\mathbb{R}^{3})}&<\infty \\
                 t^{\alpha/2\beta}\|e^{-t(-\Delta)^{\beta}}u(t)\|_{L^{q,\infty}(\mathbb{R}^{3})}&\longrightarrow0,(\hbox{as}\ t\rightarrow0).
                          \end{aligned} \right.
                          \end{equation}

\end{lemma}

\begin{proof}
Since
$u\in\widetilde{B}^{-\alpha,\infty}_{q,\infty}(\mathbb{R}^{3})$, we
have $u\in B^{-\alpha,\infty}_{q,\infty}(\mathbb{R}^{3})$. Then
$$\sup_{t>0}t^{\alpha/2\beta}\|e^{-t(-\Delta)^{\beta}}u(t)\|_{L^{q,\infty}(\mathbb{R}^{3})}<\infty$$
and there exists a sequence $\{u_{n}\}_{n\in\mathbb{N}}$ of
functions in $L^{q,\infty}(\mathbb{R}^{3})$ such that
$$\|(u_{n}-u)(t)\|_{B^{-\alpha,\infty}_{q,\infty}(\mathbb{R}^{3})}\longrightarrow \ 0,\quad \ \hbox{as}\ t\rightarrow0.$$
So there exists $N>0$ such that for $n>N$,
$$\sup_{t>0}t^{\alpha/2\beta}\|e^{-t(-\Delta)^{\beta}}(u_{n}-u)(t)\|_{L^{q,\infty}(\mathbb{R}^{3})}<\frac{\varepsilon}{2}.$$
Then for all $t>0$ we have, by Young's inequality,
$$t^{\alpha/2\beta}\|e^{-t(-\Delta)^{\beta}}u(t)\|_{L^{q,\infty}(\mathbb{R}^{3})}
<\frac{\varepsilon}{2}+t^{\alpha/2\beta}\|e^{-t(-\Delta)^{\beta}}u_{N+1}(t)\|_{L^{q,\infty}(\mathbb{R}^{3})}
<\frac{\varepsilon}{2}+Ct^{\alpha/2\beta}\|u_{N+1}(t)\|_{L^{q,\infty}(\mathbb{R}^{3})}.$$
Let
$t_{0}=\varepsilon^{2\beta/\alpha}(2C\|u_{N+1}(t)\|_{L^{q,\infty}(\mathbb{R}^{3})})^{-(2\beta/\alpha)}$,
we see that  for $t<t_{0}$,
$$t^{\alpha/2\beta}\|e^{-t(-\Delta)^{\beta}}u(t)\|_{L^{q,\infty}(\mathbb{R}^{3})}<\frac{\varepsilon}{2}+\frac{\varepsilon}{2}=\varepsilon.$$
\end{proof}

The following result gives us a condition for initial data under
which the solution to equations (\ref{eq1e}) for $\beta\in (1/2,1)$
belongs to the weak Lorentz spaces. Similar results hold for
$\beta=1,$ see Lemari$\acute{e}$-Rieusset and Prioux \cite{LRP}.

\begin{theorem}\label{th5}
Let $\frac{1}{2}<\beta<1$ and let
$\frac{2\beta}{2\beta-1}<p\leq\infty$ and
$\frac{3}{2\beta-1}<q\leq\infty$ such that
$\frac{\beta}{p}+\frac{3}{2q}=\beta-\frac{1}{2}$ and
$u_{0}\in\widetilde{B}^{-\alpha,\infty}_{L^{q,\infty}}(\mathbb{R}^{3})$
such that $\nabla\cdot u_{0}$. Then there exists $T>0$ and a mild
solution $u$ to equations (\ref{eq1e}) in the space
$\widetilde{L}^{p,\infty}((0,T),L^{q,\infty}(\mathbb{R}^{3}))$.
\end{theorem}

\begin{proof}
We construct the sequence $\{v_{n}\}_{n\in\mathbb{N}}$ as follows:
\begin{equation} \label{eq:1}
\left\{ \begin{aligned}
         v_{n}&=v_{0}-B(v_{n-1},v_{n-1})\quad\text{for } n\geq1,\\
                 v_{0}&=e^{-t(-\Delta)^{\beta}}u_{0}.
                          \end{aligned} \right.
                          \end{equation}
 We prove that for every $n\in\mathbb{N}$, the function
$v_{n}$ belongs to the space
$\widetilde{L}^{p,\infty}((0,T),L^{q,\infty}(\mathbb{R}^{3}))$. Then
we will use an induction argument on $n$.

For $n=0$, by assumption,
$u_{0}\in\widetilde{B}^{-\alpha,\infty}_{L^{q,\infty}}(\mathbb{R}^{3})$.
By Lemma \ref{le4},
\begin{equation} \label{eq:2}
\left\{ \begin{aligned}
        \sup_{t\in(0,T)}t^{1/p}\|v_{0}(t)\|_{L^{q,\infty}(\mathbb{R}^{3})}&<\sup_{t>0}t^{1/p}
        \|e^{-t(-\Delta)^{\beta}}u_{0}(t)\|_{L^{q,\infty}(\mathbb{R}^{3})}<\infty, \\
                  t^{1/p}\|v_{0}(t)\|_{L^{q,\infty}(\mathbb{R}^{3})}&=t^{1/p}
                  \|e^{-t(-\Delta)^{\beta}}u_{0}(t)\|_{L^{q,\infty}(\mathbb{R}^{3})}\longrightarrow0,\quad(\hbox{as}\ t\rightarrow0).
                          \end{aligned} \right.
                          \end{equation}
By Lemma \ref{le3}, we have
$v_{0}\in\widetilde{L}^{p,\infty}{((0,T),L^{q,\infty}(\mathbb{R}^{3}))}.$

Next we assume that
$v_{n-1}\in\widetilde{L}^{p,\infty}{((0,T),L^{q,\infty}(\mathbb{R}^{3}))}.$
Let $\varepsilon>0$, as
$v_{n-1}\in\widetilde{L}^{p,\infty}{((0,T),L^{q,\infty}(\mathbb{R}^{3}))},$
there exist two functions $v^{1}_{n-1}\in
L^{\infty}((0,T),L^{q,\infty}(\mathbb{R}^{3}))$ and $v^{2}_{n-1}\in
L^{p,\infty}((0,T),L^{q,\infty}(\mathbb{R}^{3}))$ such that
$\|v^{2}_{n-1}\|_{L^{p,\infty}((0,T),L^{q,\infty}(\mathbb{R}^{3}))}\leq\varepsilon$
and $v_{n-1}=v^{1}_{n-1}+v^{2}_{n-1}.$ We have
\begin{eqnarray*}
B(v_{n-1},v_{n-1})&=&B(v^{1}_{n-1}+v^{2}_{n-1},v^{1}_{n-1}+v^{2}_{n-1})\\
&=&B(v^{1}_{n-1},v^{1}_{n-1})+B(v^{2}_{n-1},v^{1}_{n-1})+B(v^{1}_{n-1},v^{2}_{n-1})+B(v^{2}_{n-1},v^{2}_{n-1})\\
&:=&M_{1}+M_{2}
\end{eqnarray*}
with
$$M_{1}=B(v^{1}_{n-1},v^{1}_{n-1})+B(v^{2}_{n-1},v^{1}_{n-1})+B(v^{1}_{n-1},v^{2}_{n-1})\
\ \hbox{and}\ \  M_{2}=B(v^{2}_{n-1},v^{2}_{n-1}). $$ By Lemma
\ref{le1}, we get

\begin{eqnarray*}
&&\|M_{1}\|_{L^{\infty}((0,T),L^{q,\infty}(\mathbb{R}^{3}))}\\
&\leq&\|B(v^{1}_{n-1},v^{1}_{n-1})\|_{L^{\infty}((0,T),L^{q,\infty}(\mathbb{R}^{3}))}
+\|B(v^{1}_{n-1},v^{2}_{n-1})\|_{L^{\infty}((0,T),L^{q,\infty}(\mathbb{R}^{3}))}\\
&&+\|B(v^{2}_{n-1},v^{1}_{n-1})\|_{L^{\infty}((0,T),L^{q,\infty}(\mathbb{R}^{3}))}\\
&\leq&C\|v^{1}_{n-1}\|^{2}_{L^{\infty}((0,T),L^{q,\infty}(\mathbb{R}^{3}))}
+2\|v^{1}_{n-1}\|^{2}_{L^{\infty}((0,T),L^{q,\infty}(\mathbb{R}^{3}))}\|v^{2}_{n-1}\|^{2}_{L^{p,\infty}((0,T),L^{q,\infty}(\mathbb{R}^{3}))}\\
&\leq&C.
\end{eqnarray*}
and
$\|M_{2}\|_{L^{p,\infty}((0,T),L^{q,\infty}(\mathbb{R}^{3}))}\lesssim\|v^{2}_{n-1}\|^{2}_{L^{p,\infty}((0,T),L^{q,\infty}(\mathbb{R}^{3}))}\lesssim\varepsilon^{2}.$
Thus, according to Proposition \ref{pro1}, we have
$B(v_{n-1},v_{n-1})\in\widetilde{L}^{p,\infty}((0,T),L^{q,\infty}(\mathbb{R}^{3}))$.

We will prove that for every $n\in\mathbb{N}$,
$$\|v_{n}\|_{L^{p,\infty}((0,T),L^{q,\infty}(\mathbb{R}^{3}))}\leq2\|e^{-t(-\Delta)^{\beta}}u_{0}\|_{L^{p,\infty}((0,T),L^{q,\infty}(\mathbb{R}^{3}))}.$$
Since $v_{0}=e^{-t(-\Delta)^{\beta}}u_{0}$, it is obvious that
$$\|v_{0}\|_{L^{p,\infty}((0,T),L^{q,\infty}(\mathbb{R}^{3}))}\leq2\|e^{-t(-\Delta)^{\beta}}u_{0}\|_{L^{p,\infty}((0,T),L^{q,\infty}(\mathbb{R}^{3}))}.$$
Assume that this is true for a $n\in\mathbb{N}.$ Then, we have
\begin{eqnarray*}
\|v_{n+1}\|_{L^{p,\infty}((0,T),L^{q,\infty}(\mathbb{R}^{3}))}&\leq&\|v_{0}\|_{L^{p,\infty}((0,T),L^{q,\infty}(\mathbb{R}^{3}))}
+\|B(v_{n},v_{n})\|_{L^{p,\infty}((0,T),L^{q,\infty}(\mathbb{R}^{3}))}\\
&\leq&\|e^{-t(-\Delta)^{\beta}}u_{0}\|_{L^{p,\infty}((0,T),L^{q,\infty}(\mathbb{R}^{3}))}+4C\|e^{-t(-\Delta)^{\beta}}u_{0}\|^{2}_{L^{p,\infty}((0,T),L^{q,\infty}(\mathbb{R}^{3}))}.
\end{eqnarray*}
Taking
$4C\|e^{-t(-\Delta)^{\beta}}u_{0}\|_{L^{p,\infty}((0,T),L^{q,\infty}(\mathbb{R}^{3}))}<1$,
we get
$$\|v_{n+1}\|_{L^{p,\infty}((0,T),L^{q,\infty}(\mathbb{R}^{3}))}\leq2\|e^{-t(-\Delta)^{\beta}}u_{0}\|_{L^{p,\infty}((0,T),L^{q,\infty}(\mathbb{R}^{3}))},$$
that is,
$\|v_{n+1}\|_{L^{p,\infty}((0,T),L^{q,\infty}(\mathbb{R}^{3}))}$ in
the ball centered at $0$, of radius
$2\|e^{-t(-\Delta)^{\beta}}u_{0}\|_{L^{p,\infty}((0,T),L^{q,\infty}(\mathbb{R}^{3}))}.$
Then,
\begin{eqnarray*}
&&\|v_{n}-v_{n-1}\|_{L^{p,\infty}((0,T),L^{q,\infty}(\mathbb{R}^{3}))}\\
&\leq&\|B(v_{n-1}-v_{n-2},v_{n-1})\|_{L^{p,\infty}((0,T),L^{q,\infty}(\mathbb{R}^{3}))}+\|B(v_{n-2},v_{n-1}-v_{n-2})\|_{L^{p,\infty}((0,T),L^{q,\infty}(\mathbb{R}^{3}))}\\
&\leq&
C\|v_{n-1}-v_{n-2}\|_{L^{p,\infty}((0,T),L^{q,\infty}(\mathbb{R}^{3}))}
\left(\|v_{n-1}\|_{L^{p,\infty}((0,T),L^{q,\infty}(\mathbb{R}^{3}))}+\|v_{n-2}\|_{L^{p,\infty}((0,T),L^{q,\infty}(\mathbb{R}^{3}))}\right)\\
&\leq&
4C\|e^{-t(-\Delta)^{\beta}}u_{0}\|_{L^{p,\infty}((0,T),L^{q,\infty}(\mathbb{R}^{3}))}\|v_{n-1}-v_{n-2}\|_{L^{p,\infty}((0,T),L^{q,\infty}(\mathbb{R}^{3}))}.
\end{eqnarray*}
 Thus, the Picard contraction principle completes the proof.
\end{proof}

Now, we want to give   the reverse result of Theorem \ref{th5}. To
do this, we need the following lemma.

\begin{lemma}\label{le6}
Let $\frac{1}{2}<\beta< 1$ and let $\frac{3}{2\beta-1}<q<\infty$ and
$\frac{2\beta}{2\beta-1}<p<\infty$ such that
$\frac{2\beta}{(2\beta-1)p}+\frac{3}{(2\beta-1)q}=1$ and
$u\in\widetilde{L}^{p,\infty}((0,T),L^{q,\infty}(\mathbb{R}^{3}))$
be a mild solution to equations (\ref{eq1e}). Then, for
$0<\varepsilon<1$, there exists $0<t_{0}<T$ such that $\forall
t\in(0,t_{0}]$,
$\|u(t)\|_{L^{q,\infty}(\mathbb{R}^{3})}\leq\frac{\varepsilon}{2C_{0}t^{1/p}}$.
\end{lemma}

\begin{proof}
It follows from
$u\in\widetilde{L}^{p,\infty}((0,T),L^{q,\infty}(\mathbb{R}^{3}))$
that   for all $\lambda>0$, there exists a constant $C(\lambda)$,
depending on $\lambda$, such that $C(\lambda)\rightarrow0$
$(\hbox{as}\ \lambda\rightarrow\infty)$ and
$$\left|\left\{\|u(t)\|_{L^{q,\infty}(\mathbb{R}^{3})}>\lambda\right\}\right|<\frac{C(\lambda)}{\lambda^{p}}.$$
Let $0<\varepsilon<1$ and $0<t_{0}<1$. Denote
$\lambda_{t_{0}}=\frac{\varepsilon}{4C_{0}t_{0}^{1/p}}$. When
$t_{0}\rightarrow0$ and $C(\lambda_{t_{0}})\rightarrow0$, we choose
$t_{0}$ small enough such that
$C(\lambda_{t_{0}})<\frac{\varepsilon^{p}}{2\times4^{2p}C_{0}^{p}}.$
Let $t\leq t_{0}$ such that
$\lambda_{t}=\frac{\varepsilon}{4C_{0}t^{1/p}}\geq\lambda_{t_{0}}$,
then $C(\lambda_{t})\leq C(\lambda_{t_{0}})$. We can get
\begin{equation}\label{eq9}
\left|\left\{t\in(0,T),\|u(t)\|_{L^{q,\infty}(\mathbb{R}^{3})}>\lambda_{t}\right\}\right|<\frac{C(\lambda_{t})}{\lambda_{t}^{p}}
<\frac{t}{2\times 4^{p}}.
\end{equation}
We claim that there exists $\theta$ such that
$$t-\frac{t}{4^{p}}\leq\theta\leq  t\quad\text{and}\quad\|u(\theta)\|_{L^{q,\infty}(\mathbb{R}^{3})}\leq\frac{\varepsilon}{4C_{0}t^{1/p}}=\lambda_{t}.$$
Otherwise
$$\left|\left\{t\in(0,T),\|u(t)\|_{L^{q,\infty}(\mathbb{R}^{3})}>\lambda_{t}\right\}\right|\geq\left|[t-\frac{t}{4^{p}},t]\right|=\frac{t}{4^{p}}.$$
This is a contraction to (\ref{eq9}). Let
$T^{\ast}=(4C_{0}\|u(\theta)\|_{L^{q,\infty}(\mathbb{R}^{3})})^{-p}$.
Taking $0<\varepsilon<1$, we have
\begin{equation}\label{eq10}
\|u(\theta)\|_{L^{q,\infty}(\mathbb{R}^{3})}\leq\frac{1}{4C_{0}t^{1/p}}\Longrightarrow
t\leq(4C_{0}\|u(\theta)\|_{L^{q,\infty}(\mathbb{R}^{3})})^{-p}.
\end{equation}
 Applying
Lemma \ref{le2} in the interval $[\theta,\theta+T^{\ast}]$, there
exists  a solution $\tilde{u}\in
L^{\infty}((\theta,\theta+T^{\ast}),L^{q,\infty}(\mathbb{R}^{3}))$
to the  equations (\ref{eq1e}). Note that (\ref{eq10}) implies that
$$(\theta,t]\subset(\theta,\theta+t)\subset(\theta,\theta+T^{\ast}).$$
By Proposition \ref{pro2}, we know  $u=\tilde{u}$ on $(\theta,t]$.
So for $t\leq t_{0}$, there exists $0<\theta<t$ such that $u\in
L^{\infty}((\theta,t],L^{q,\infty}(\mathbb{R}^{3}))$ and
$$\forall s\in(\theta,t], \|u(s)\|_{L^{q,\infty}(\mathbb{R}^{3})}\leq2\|u(\theta)\|_{L^{q,\infty}(\mathbb{R}^{3})}\leq\frac{\varepsilon}{2C_{0}t^{1/p}}.$$
This completes the proof of this lemma.
\end{proof}

\begin{theorem}\label{the7}
For $\frac{1}{2}<\beta< 1$ and let $\frac{3}{2\beta-1}<q<\infty$ and
$\frac{2\beta}{2\beta-1}<p<\infty$ such that
$\frac{2\beta}{(2\beta-1)p}+\frac{3}{(2\beta-1)q}=1$ and
$u\in\widetilde{L}^{p,\infty}((0,T),L^{q,\infty}(\mathbb{R}^{3}))$
be a mild solution to equations (\ref{eq1e}). Then
\begin{equation} \label{eq11}
\left\{ \begin{aligned}
        &\sup_{t\in(0,T)}t^{1/p}\|u(t)\|_{L^{q,\infty}(\mathbb{R}^{3})}<\infty, \\
                  &t^{1/p}\|u(t)\|_{L^{q,\infty}(\mathbb{R}^{3})}\longrightarrow0\quad(\hbox{as}\ t\rightarrow0).
                          \end{aligned} \right.
                          \end{equation}
\end{theorem}

\begin{proof}
 By Lemma \ref{le6}, for every $\varepsilon>0$, there exists $t_{0}$ such
 that, for all $t\in (0,t_{0}),$
 $$t^{1/p}\|u(t)\|_{L^{q,\infty}(\mathbb{R}^{3})}\leq\frac{\varepsilon}{2C_{0}},$$
that is,
$\lim_{t\rightarrow0}t^{1/p}\|u(t)\|_{L^{q,\infty}(\mathbb{R}^{3})}=0.$

 Now we
prove the first assertion of (\ref{eq11}). Checking the proof of
Lemma \ref{le6} and taking $\varepsilon=\frac{1}{2}$, we can see
that there exist $t_{0}$ such that
 for every $t\leq t_{0}$ and  $0<\theta<t$ such that $u\in
L^{\infty}((\theta,t],L^{q,\infty}(\mathbb{R}^{3}))$ and
\begin{equation}\label{eq12} \forall
s\in(\theta,t],\quad
\|u(s)\|_{L^{q,\infty}(\mathbb{R}^{3})}\leq2\|u(\theta)\|_{L^{q,\infty}(\mathbb{R}^{3})}\leq\frac{1}{4C_{0}t^{1/p}}.
\end{equation}
 On the other hand, Lemma \ref{le6} and
$\lim_{t\rightarrow0}t^{1/p}\|u(t)\|_{L^{q,\infty}(\mathbb{R}^{3})}=0$
 tell us that there exists $t_{1}$ such that for $s\in(0,t_{1})$,
$t^{1/p}\|u(t)\|_{L^{q,\infty}(\mathbb{R}^{3})}\leq C$. If
$t_{0}>t_{1}$, take $t_{2}<t_{1}<t_{0}$ (otherwise take
$t_{2}=t_{0}$). By (\ref{eq12}), there exists $\theta_{2}$ such that
for every $s\in (\theta_{2},t_{2}]$,
$\|u(s)\|_{L^{q,\infty}(\mathbb{R}^{3})}\leq\frac{1}{4C_{0}t_{2}^{1/p}}$.
Because $t^{1/p}\|u(t)\|_{L^{q,\infty}(\mathbb{R}^{3})}$ is bounded
on $(0,\theta_{2}]\subset (0,t_{1}),$ now we restrict
$t\in(\theta_{2},T]$. Define a new function
$\tilde{u}(s)=u(t_{2}-\theta_{2}+s)$. Then we only need to prove the
assertion for $\tilde{u}(s)$ on
$s\in(\theta_{2},T+t_{2}-\theta_{2}]$.

Since $u$ is a solution to equations (\ref{eq1e}),
$$u\in\widetilde{L}^{p,\infty}((t_{2},T),L^{q,\infty}(\mathbb{R}^{3}))
\Longrightarrow\tilde{u}\in\widetilde{L}^{p,\infty}((\theta_{2},T-t_{0}+\theta_{2}),L^{q,\infty}(\mathbb{R}^{3}))$$
implies that $\tilde{u}$ is also a solution to the  equations
(\ref{eq1e}). By Lemma \ref{le6} for $\varepsilon=\frac{1}{2}$
again, we can get that for $\forall t\in(\theta_{2},t_{2})$,
$\|\tilde{u}(t)\|_{L^{q,\infty}(\mathbb{R}^{3})}\leq\frac{1}{4C_{0}t_{2}^{1/p}}$.
That is,  $\forall t\in(t_{2}, 2t_{2}-\theta_{2})$,
$\|u(t)\|_{L^{q,\infty}(\mathbb{R}^{3})}\leq\frac{1}{4C_{0}t^{1/p}_{2}}$.
We conclude that
$$\forall t\in(\theta_{2},2t_{2}-\theta_{2}),\quad \|u(t)\|_{L^{q,\infty}(\mathbb{R}^{3})}\leq\frac{1}{4C_{0} t^{1/p}_{2}}.$$
Since $T$ is finite, we can find a constant $n\in \mathbb{N}$ such
that $nt_{2}<T<(n+1)t_{2}$. Hence repeating this argument  finite
many  times, we get
$$\forall t\in(\theta_{2}, T], \quad \|u(t)\|_{L^{q,\infty}(\mathbb{R}^{3})}\leq\frac{1}{4C_{0}t^{1/p}_{2}}<\frac{1}{4C_{0}t^{1/p}_{2}}\frac{T^{1/p}}{t^{1/p}}.$$
This completes the proof of this theorem.
\end{proof}
\vspace{0.1in} \noindent
 {\bf{Acknowledgements.}} We would like to thank our  supervisor Professor Jie Xiao
  for suggesting the problem and kind  encouragement.


\begin{thebibliography}{10}


\bibitem{M Cannone}
M. Cannone, \textit{A generalization of a theorem by Kato on
Navier-Stokes equations,} {Rev. Mat. Iberoam.,} \textbf{13} (1997),
673-697.

\bibitem{Cannone}
M. Cannone, \textit{Harmonic analysis tools for solving the
incompressible Navier-Stokes equations}, {In:  Handbook of
Mathematical Fluid Dynamics} {\bf{Vol} 3}(eds. S. Friedlander, D.
Serre), Elsevier, 2004, pp. 161-244.


\bibitem{G. Dafni J. Xiao}
G. Dafni and J. Xiao, \textit{Some new tent spaces and duality
theorem for fractional Carleson measures and
$Q_{\alpha}(\mathbb{R}^{n})$,} {J. Funct. Anal.,} \textbf{208}
(2004), 377-422.

\bibitem{G. Dafni J. Xiao 1}
G. Dafni and J. Xiao, \textit{The dyadic structure and atomic
decomposition of $Q$ spaces in several varialbes,} {Tohoku Math.
J.,} \textbf{57} (2005), 119-145.


\bibitem{M. Essen S. Janson L. Peng J. Xiao}
M. Essen, S. Janson, L. Peng  and J. Xiao, \textit{$Q$ space of
several real variables}, {Indiana Univ. Math. J.,} \textbf{49}
(2000), 575-615.


\bibitem{Y Giga T Miyakawa}
Y. Giga, T. Miyakawa, \textit{Navier-Stokes flow in $\mathbb{R}^{3}$
with measures as initial vorticity and Morry spaces,} {Comm. Partial
Differential Equtions,} \textbf{14} (1989), 577-618.


\bibitem{T Kato}
T. Kato, \textit{Strong $L^{p}-$solutions of the Navier-Stokes in
$\mathbb{R}^{n}$ with applications to weak solutions,} {Math.
Zeit.,} \textbf{187} (1984), 471-480.



\bibitem{H. Koch D. Tataru}
H. Koch and D. Tataru, \textit{Well-posedness for the Navier-Stokes
equations}, {Adv. Math.,} \textbf{157} (2001), 22-35.


\bibitem{LR}
P. G. Lemari$\acute{e}$-Rieusset, \textit{Recent Development in the
Naiver-Stokes Problem, in: Research Notes in Mathematics,}
\textbf{431}, {Chapman-Hall/CRC,} 2002.

\bibitem{LRP}
P. G. Lemari$\acute{e}$-Rieusset and N. Prioux, \textit{The
Naiver-Stokes equations with data in $bmo^{-1}$,} {Nonlinear
Analysis} \textbf{70} (2009), 280-297.

\bibitem{Leray}
L. Leray, \textit{Sur le mouvement d'un liquide visqueux emplissant
l'espace}, {Acta Math.,} \textbf{63} (1934), 193-248.


\bibitem{J L Lions}
J. L. Lions, \textit{Quelques m\'{e}thodes de r\'{e}solution des
probl\`{e}mes aux limites non lin\'{e}aires,} (French) {Paris:
Dunod/Gauthier-Villars}, 1969.

\bibitem{Li Zhai}
P. Li and Z. Zhai, \textit{Well-posedness and Regularity of
 Generalized Naiver-Stokes Equations in Some Critical Q-spaces,}
{Submitted.}








\bibitem{J Wu 2}
J. Wu, \textit{The generalized incompressible Navier-Stokes
equations in Besov spaces}, {Dyn. Partial Differ. Eq.,} \textbf{1}
(2004), 381-400.

\bibitem{J Wu 3}
J. Wu, \textit{Lower Bounds for an integral involving fractional
Laplacians and the generalized Navier-Stokes equations in Besov
spaces,} {Commun. Math. Phys.,} \textbf{263} (2005), 803-831.


 \bibitem{J. Xiao 1}
J. Xiao,  \textit{Homothetic variant of fractional Sobolev space
with application to Navier-Stokes system}, {Dynamic of PDE.,}
\textbf{2} (2007), 227-245.


\end{thebibliography}
\end{document}